\documentclass[11pt]{amsart}
\usepackage{amsmath}
\usepackage{amssymb}
\usepackage{amsthm}
\usepackage{mathrsfs}
\usepackage[all,cmtip]{xy}
\usepackage{times}
\usepackage{multicol}
\usepackage{fancyhdr}
\def\Z{{\mathbb Z}}
\def\N{{\mathbb N}}
\def\G{{\mathscr G}}

\def\I{{\mathscr I}}

\def\ra{\rightarrow}

\def\ep{\varepsilon}
\def\ph{\varphi}

\def\<{\langle}
\def\>{\rangle}
\def\leq{\leqslant}
\def\geq{\geqslant}
\def\st{\operatorname{star}}
\def\obj{\operatorname{obj}}
\def\fis{{\it FIS}}

\def\id{\operatorname{id}}

\def\der{\textbf{Der}}

\def\homogpd{\textbf{OGpd}}

\def\ol{\overline}

\def\dom{\mathbf d}
\def\ran{\mathbf r}
\def\ogpd{{\sf OGPD}}

\def\twH{\widetilde{H}_{\phi}}
\def\twHpsi{\widetilde{H}_{\psi}}
\def\act{\operatorname{Act}}
\def\cov{\operatorname{Cov}}
\def\br{\operatorname{BR}}

\newcommand{\twobar}{/\kern-0.3em/}

\numberwithin{equation}{section}

\title{Fibrations of ordered groupoids and the factorization of ordered functors}

\author{Nouf AlYamani,
N. D. Gilbert, and E.C. Miller.}
\address{
School of Mathematical and Computer Sciences\\
and the Maxwell Institute for the Mathematical Sciences,\\
Heriot-Watt University, Edinburgh EH14 4AS, U.K}
\email{naaa4@hw.ac.uk, N.D.Gilbert@hw.ac.uk, ecm003@hotmail.com}

\thanks{The first author acknowledges the support of a PhD scholarship from the Royal Government of Saudi Arabia, held whilst this work was carried out.
The third author acknowledges the financial support of a studentship from the Carnegie Trust for the Universities of Scotland.}

\date{}
\newtheorem{theorem}{Theorem}[section]
\newtheorem{prop}[theorem]{Proposition}
\newtheorem{lemma}[theorem]{Lemma}
\newtheorem{cor}[theorem]{Corollary}

\theoremstyle{definition}
\newtheorem{example}[theorem]{Example}
\newtheorem{defn}[theorem]{Definition}
\newtheorem{remark}[theorem]{Remark}

\begin{document}

\begin{abstract}
We investigate canonical factorizations of ordered functors of ordered groupoids through star-surjective functors.  Our main construction is a 
quotient ordered groupoid, depending on an ordered version of the notion of normal subgroupoid, that results is the factorization of an ordered
functor as a star-surjective functor followed by a star-injective functor.  Any star-injective functor possesses a universal factorization through a
covering, by Ehresmann's Maximum Enlargement Theorem.  We also show that any ordered functor has a canonical factorization through a functor with the
ordered homotopy lifting property.
\end{abstract}

\subjclass[2010]{Primary: 20L05; Secondary 18D15,55U35}

\keywords{ordered groupoid, functor, fibration, covering}

\maketitle

\section*{Introduction}
The idea of treating a groupoid -- a small category in which every arrow is invertible -- as an algebraic structure 
has been established as a profitable approach to problems in group theory, where groupoids offer algebraic analogues
of relevant notions from the homotopy theory of paths and covering spaces, and applications of group theory can often be more
conveniently phrased in terms of groupoids.  The book \cite{HiBook} by P.J.Higgins is an admirable introduction.    The role of groupoids in 
describing symmetry is concisely surveyed in \cite{Wei}.

In the 1950's,
C. Ehresmann developed a new foundation of differential geometry based on groupoids \cite{Ehres}.  His algebraic model of a pseudogroup of transformations
was an \emph{ordered} groupoid, in which the key structural idea corresponds to the restriction of domain.  The introduction of an ordered structure
leads to a separation from group theory since a group, when considered as an ordered groupoid, can only carry the equality ordering.  Instead, the
most closely related algebraic structure is that of an \emph{inverse semigroup}.  Inverse semigroups correspond to the subclass of inductive groupoids,
in which the poset of identity elements forms a semilattice.  The consequent close interactions between the theories of ordered groupoids and inverse semigroups
are  a major theme of the book \cite{LwBook} of Lawson.

In this paper, we consider some of the homotopy notions modelled by groupoids, but in the context of ordered groupoids.   Our starting point is the paper \cite{Br} on
fibrations of groupoids by  Brown, in which fundamental constructions for fibrations of groupoids are explored, and applications given to the non-abelian cohomology of groups and to
track groupoids in homotopy theory.  The groupoid notions of the homotopy-lifting property and the path-lifting property for a map of groupoids $p:G \ra H$ are shown to be equivalent, and equivalent to $p$ being star-surjective.  Any one of these properties thus defines $p$ as a \emph{fibration}.   Steinberg \cite{St} adopts star-surjectivity as the definition of a fibration
of ordered groupoids.  He undertakes a general study of factorization theorems for ordered groupoids, based on the semidirect product construction for one ordered groupoid acting on another
(also found in \cite{Br2}) and a left adjoint $\der$ -- the derived ordered groupoid -- to the semidirect product.  Steinberg uses the derived ordered groupoid to show that every
ordered groupoid morphism factors in a canonical way through a fibration, and that a star-injective morphism has a universal factorization through a covering.  The latter 
result is Ehresmann's Maximum Enlargement Theorem \cite{Ehres}.  Lawson's approach to this result \cite{LwBook} also demonstrates its applications in inverse semigroup
theory, and in particular to the theory of idempotent pure homomorphisms.
Another important precursor to our work is the thesis of Matthews \cite{Mthesis} and the paper \cite{LMP}.  Matthews studies inverse semigroups and ordered groupoids from
the point of view of abstract homotopy theory, showing that the category of ordered groupoids admits cocylinders, and as an application is able to deduce Steinberg's
Fibration Theorem.

This paper focusses on fibrations of ordered groupoids, and our first significant result is that in the ordered case, the homotopy lifting property and star-surjectivity
are no longer equivalent properties.  Star-surjectivity is still equivalent to path-lifting, but homotopy lifting is a stronger property: we therefore consider star-surjective
ordered functors -- \emph{fibrations} as in \cite{Br,St} -- and \emph{strong} fibrations that have the homotopy lifting property.  We establish some formal properties of these
classes of fibration in section \ref{fibrations}, and show in section \ref{quotog} how a quotient construction for ordered groupoids leads to a factorization of any morphism as a fibration followed by a star-injective
functor.   This is an extension to the ordered case of a basic result for groupoid morphisms \cite[Proposition 2.4]{Br}, but in the ordered case requires a remodelled
definition of a \emph{normal} ordered subgroupoid $A$ of an ordered groupoid $G$ (more general than that of \cite{Mthesis}) and the construction of a
quotient ordered groupoid $G \twobar A$, for which the canonical map $G \ra G \twobar A$ is always a fibration.

We then consider the Maximum Enlargement Theorem in section \ref{maxenlthm} and the Fibration Theorem in section \ref{fibthm}.  We give a direct proof of the
Maximum Enlargement Theorem, by constructing from a star-injective morphism $U \ra H$ a poset $P$ with a natural $H$--action such that the semidirect product
$H \ltimes P$ is an enlargement of $U$.  Coupled with our quotient construction, the Maximum Enlargement Theorem then gives a universal factorization of any
ordered groupoid morphism as a fibration followed by an enlargement followed by a covering.  We follow \cite{LMP} and use the mapping cocylinder $M^{\phi}$ of an ordered groupoid morphism $\phi: G \ra H$ in our discussion of the Fibration Theorem, but we show directly that the morphism $M^{\phi} \ra H$ is a strong fibration.  The factorization
of an ordered groupoid morphism $\phi: G \ra H$ as $G \ra M^{\phi} \ra H$ then implies that $\phi$ factorizes as an enlargement followed by a fibration followed by a covering.

\section{Groupoids and ordered groupoids}
\label{gpdogpd}
A {\em groupoid} $G$ is a small category in which every morphism
is invertible.  We consider a groupoid as an algebraic structure
(as in \cite{HiBook,LwBook}): the elements are the morphisms, and
composition is an associative partial binary operation.
The set of identities in $G$ is denoted $G_0$, and an element $g \in G$ has domain
$g\dom=gg^{-1}$ and range $g\ran=g^{-1}g$.  If every arrow in $G$ is an identity, we say that
$G$ is {\em trivial}.  The groupoid $\I$ has two identities, $0$ and $1$, and two non-identity morphisms,
$\iota$ and $\iota^{-1}$, with $\iota \dom = 0$ and $\iota \ran =1$.

Let $e \in G_0$.  Then the {\em star} of $e$ in $G$ is the set
$\st_G(e) = \{ g \in G : g \dom =e \}$, and the {\em local group} at $e$ is the set
$G(e) = \{ g \in G : g \dom = e = g \ran \}$.  
A functor
$\phi : G \ra H$ is said to be {\em star-surjective} if, for each $e \in G_0$,
the restriction $\phi : \st_G(e) \ra \st_H(e \phi)$ is surjective.
{\em Star-injective} and {\em star-bijective} functors are defined
similarly.   A star-surjective functor is also called a \emph{fibration}, a star-injective functor an \emph{immersion},
and a star-bijective functor a \emph{covering}.

The kernel of a functor $\phi$ is the inverse image of the identities:
\[ \ker \phi = \{ g \in G : g \phi \in H_0 \}. \]

\begin{lemma}
A functor $\phi: G \ra H$ is star-injective if and only if its kernel is equal to $G_0$.
\end{lemma}

A \emph{normal} subgroupoid $N$ of a groupoid $G$ is a subgroupoid such that
\begin{itemize}
\item $N$ is wide in $G$: that is , $N_0=G_0$,
\item if $g \in G$ and $n \in N$ with $g^{-1}ng$ defined in $G$, then $g^{-1}ng \in N$.
\end{itemize}
We remark that $g^{-1}ng$ is defined if and only if  $n \in G(g \dom)$.  Now $N$ may well contain elements not in local groups, and
this definition follows that given by Higgins in \cite{HiBook}.   However, we note that Matthews \cite{Mthesis} assumes that
if $n \in G$ then $n \in G(e)$ for some $e \in G_0$.  This distinction will be of importance in section \ref{quotog}.
A normal subgroupoid $N$ then determines an equivalence relation $\simeq_N$ on $G$, defined by
\[ g \simeq_N h \Longleftrightarrow \; \text{there exist} \; m,n \in N \; \text{such that} \; g = mhn \,. \]
The set $G/N$ of $\simeq_N$--equivalence classes inherits a natural groupoid structure, see \cite{HiBook}.  The kernel of a functor $\phi: G \ra H$ is a 
normal subgroupoid, and the canonical map $G \ra G/ \ker \phi$ is a fibration, and $\phi$ then induces an immersion $G/ \ker \phi \ra H$.

An {\em ordered groupoid} $(G,\leq)$ is a groupoid $G$
with a partial order $\leq$ satisfying the following axioms:

\begin{enumerate}
\item[OG1] for all $g,h \in G$, if $g \leq h$ then $g^{-1} \leq h^{-1}$,
\item[OG2] if $g_1 \leq g_2 \,, h_1 \leq h_2$ and if the compositions
$g_1h_1$ and $g_2h_2$ are   defined, then $g_1h_1 \leq g_2h_2$,
\item[OG3] if $g \in G$ and $f$ is an identity of $G$ with $f \leq
g \dom$, there exists a unique element $(f|  g)$, called the {\em restriction}
of $g$ to $f$, such that $(f|g)\dom=f$ and $(f|g) \leq g$.
\end{enumerate}
As a consequence of [OG3] we also have:
\begin{enumerate}
\item[OG3*] if $g \in G$ and $f$ is an identity of $G$ with $f \leq
g \ran$, there exists a unique element $(g|  f)$, called the {\em corestriction}
of $g$ to $f$, such that $(g|f)\ran=f$ and $(g|  f) \leq g$,
\end{enumerate}
since the corestriction of $g$ to $f$ may be defined as $(f |   g^{-1})^{-1}$.

Let $G$ be an ordered groupoid and let $a,b \in G$.  Suppose that
$a \ran$ and $b \dom$ have a greatest lower bound -- which we
write as $(a \ran)(b \dom)$ -- in $G_0$.  Then we may define the
{\em pseudoproduct} of $a$ and $b$ in $G$ as
$$a \ast b = (a |  (a \ran)(b \dom))((a \ran)(b \dom) |  b),$$
where the right-hand side is a composition in the groupoid $G$. As 
Lawson shows in Lemma 4.1.6 of \cite{LwBook}, this is a partially
defined associative operation on $G$. 

Our discussion of ordered groupoids follows the terminology and notation
of \cite{LwBook}, except that we have interchanged the use of domain
and range and that we use $\ast$ to denote the pseudoproduct, rather
than $\otimes$.

An ordered groupoid $G$ is \emph{inductive} if $G_0$ is a meet semilattice.  In this case the
pseudoproduct is everywhere defined on $G$, and makes $G$ into an inverse semigroup.  On the
other hand, each inverse semigroup $S$ has an underlying inductive groupoid $G(S)$, whose arrows are the elements of $S$,
with $s\dom=ss^{-1}$ and $s\ran=s^{-1}s$, and where composition is just the restriction of the multiplication in $S$.
These constructions yield an isomorphism between the categories of inductive groupoids and of inverse semigroups,
see \cite[Theorem 4.1.8]{LwBook}.

\begin{example}
\label{presheaf}
Let $P$ be a poset and $G = (\G,P)$ be a presheaf of groups over $P$, with linking maps $\alpha^x_y$ for $x \geq y$ in $P$.  Here $\G$ assigns a group $G_x$ to
each $x \in P$ and $\alpha^x_y$ is a group homomorphism $G_x \ra G_y$.  We have $\alpha^x_x = \id$ for all $x \in P$, and whenever $x \geq y \geq z$ in $P$ then
$\alpha^x_y  \alpha^y_z  = \alpha^x_z$.  The groupoid composition on $G$ is just the collection of group operations in the groups $G_x$, and $G$  is ordered by $g \geq g \alpha^x_y$ whenever $x,y \in P$ with $x \geq y$ and $g \in G_x$.
\end{example}

A morphism of ordered groupoids is an order-preserving functor $A \ra B$.  We shall call
such functors simply \emph{ordered} functors.  We then have the category $\homogpd$ of ordered
groupoids and ordered functors.  The set of ordered functors $A \ra B$ is denoted by $\homogpd(A,B)$.
We note that the category of ordered groupoids is cartesian closed.  
So given ordered groupoids $A,B$ there exists an ordered groupoid
$\ogpd(A,B)$ such that there is a natural equivalence 
\begin{equation}
\label{exp_law}
\ogpd(A \times B,C) \ra \ogpd(B,\ogpd(A,C)) \,.
\end{equation}
The object set of $\ogpd(A,B)$ is $\homogpd(A,B)$.  Given two ordered functors $f,g : A \ra B$,
an arrow in $\ogpd(A,B)$ from $f$ to $g$ is an ordered natural transformation $\tau$ from $f$ to $g$: that is,
$\tau$ is an ordered function $\obj(A) \ra B$ such that, for each arrow $a \in A$ with $a \dom = x$ and $a \ran =y$,
the square
\[ \xymatrixcolsep{3pc}
\xymatrix{
xf \ar[d]_{x \tau} \ar[r]^{af} & yf \ar[d]^{y \tau}\\
xg \ar[r]_{ag} & yg}
\]
in $B$ commutes.

\section{Fibrations}
\label{fibrations}
Let $p: G \ra H$ be a morphism of groupoids.  Then $p$ has the {\em homotopy lifting property}
if, given any groupoid $A$ and a commutative square
\[ \xymatrixcolsep{3pc}
\xymatrix{
A \ar[d]_{i_0} \ar[r]^{f} & G \ar[d]^{p}\\
A \times \I \ar[r]_{F} & H}
\]
where $ai_0=(a,0)$, 
there exists a morphism $\widetilde{F} : A \times \I \ra G$ such that the diagram
\[ \xymatrixcolsep{3pc}
\xymatrix{
A \ar[d]_{i_0} \ar[r]^{f} & G \ar[d]^{p}\\
A \times \I \ar[r]_{F} \ar[ur]_{\widetilde{F}} & H}
\]
commutes.  Proposition 2.1 of Brown's paper \cite{Br} shows that possession of this property by $p$ is equivalent to $p$ being star-surjective, and so drawing on the evident topological analogy, Brown calls star-surjective morphisms {\em fibrations of groupoids}.  In fact,
\cite[Proposition 2.1]{Br} shows the equivalence of the homotopy lifting property, the path-lifting property, and star-surjectivity.
The path-lifting property is the special case of the homotopy lifting property in which the groupoid $A$ is a trivial groupoid with
one identity.

The homotopy lifting property and the path-lifting property for ordered groupoids are formulated by repeating the above definitions with the addition that every morphism should be a morphism of ordered groupoids.  A distinction then arises between the homotopy lifting property and being star-surjective.

\begin{prop}{\cite[Proposition 4.7]{LMP}}
A morphism of ordered groupoids has the path lifting property if and only if it is star-surjective: in particular, a 
morphism of ordered groupoids with the homotopy lifting property is star-surjective.
\end{prop}

\begin{proof}
As already remarked, the path-lifting property is a special case of the homotopy lifting property.
So suppose that $p:G \ra H$ has the path lifting property. Let $e$ be an identity of $G$ and suppose
that $h \in \st_H(ep)$. Let $0$ denote a trivial groupoid with a single identity. Any morphism from
$0 \times \I \ra H$ is determined by the choice of a single arrow in $H$, and will be ordered. Thus $e$ determines a 
morphism $0 \ra G$ and $h$ determines a morphism $0 \times \I \ra H$ making the commutative square 
\[ \xymatrixcolsep{3pc}
\xymatrix{
0 \ar[d]_{i_0} \ar[r]^{f} & G \ar[d]^{p}\\
0 \times \I \ar[r]_{F} & H \,.}
\]
The lifting $\widetilde{F} : 0 \times A \ra G$ in the commutative diagram 
\[ \xymatrixcolsep{3pc}
\xymatrix{
0 \ar[d]_{i_0} \ar[r]^{f} & G \ar[d]^{p}\\
0 \times \I \ar[r]_{F} \ar[ur]_{\widetilde{F}} & H}
\]
is then determined by an arrow $g \in G$ satisfying $g \dom = e$ and $gp=h$.  Therefore $p$ is star-surjective.

For the converse, given a commutative square
\[ \xymatrixcolsep{3pc}
\xymatrix{
0 \ar[d]_{i_0} \ar[r]^{f} & G \ar[d]^{p}\\
\I \ar[r]_{F} & H}
\]
with $p$ star-surjective, we obtain a morphism $\widetilde{F}: \I \ra G$ making
\[ \xymatrixcolsep{3pc}
\xymatrix{
0 \ar[d]_{i_0} \ar[r]^{f} & G \ar[d]^{p}\\
\I \ar[r]_{F} \ar[ur]_{\widetilde{F}} & H}
\]
commute by defining $\iota \widetilde{F}$ to be any arrow in $\st_G(0f)$ mapping to $\iota F \in H$.  Since $\I$ carries
the trivial ordering, any choice for $\iota \widetilde{F}$ determines an ordered morphism.
\end{proof}

However, a star-surjective ordered morphism need not have the ordered homotopy lifting property, as the following
counterexample shows.

\begin{example}
\label{not_ord_fib}
Let $E = \{e,f,z\}$ be the semilattice with $e \geq z \leq f$.  We take $G$ to be a semilattice of groups over $E$
with linking maps $\alpha^e_z : G_e \ra G_z$ and $\alpha^f_z : G_f\ra G_z$, and we take $H$ to be a semilattice of
groups over the semilattice $\{0,1\}$, where $0<1$, with linking map $\psi=\psi^1_0 : H_1 \ra H_0$.  Let $i:E \ra G$
be the inclusion map, and let $p:G \ra H$ be the star-surjective morphism determined by three surjections
$p_e:G_e \ra H_1, p_f:G_f \ra H_1$ and $p_z:G_z \ra H_0$.  Hence the following diagram commutes:
\begin{equation} \label{GHmorph}
 \xymatrixcolsep{3pc}
\xymatrix{
G_e \ar@{>>}[drrr]_-{p_e} \ar[ddr]^{\alpha^e_z} && G_f \ar@{>>}[dr]^{p_f} \ar[ddl]^{\alpha^f_z}|(.28)\hole &\\
& & & H_1 \ar[dd]^{\psi}\\
& G_z \ar@{>>}[rrd]_{p_z} && \\
&&& H_0}
\end{equation}

A homotopy $F: E \times \I \ra H$ is determined by two elements $h_e,h_f \in H$ satisfying $h_e\psi = h_f\psi$,
and then $(e,\iota)F=h_e, (f,\iota)F=h_f$ and $(z,\iota)F=h_e \psi=h_f\psi$.  

To construct the lifting homotopy $\widetilde{F}: E \times \I \ra G$ we are required to find $g_e,g_f \in G$
such that 
\begin{itemize}
\item $g_e \alpha^e_z = g_f \alpha^f_z$
\item $g_ep_e=h_e$ and $g_fp_f=h_f$
\item $g_e \alpha^e_z p_z = g_f \alpha^f_z p_z = h_e \psi = h_f \psi$.
\end{itemize}
Only the first of these conditions is imposed by the requirement that $\widetilde{F}$ be ordered.  To fulfill the
second and third conditions we can choose any preimages $g_e, g_f$ of $h_e,h_f$ under $p_e,p_f$ to satisfy the
second condition, and commutativity of \eqref{GHmorph}  guarantees the third condition.  

However, we cannot always make a consistent choice of $g_e,g_f$ to satisfy all three conditions.  Take
$G_e = \{1,a\}, G_f = \{ 1,b \}$ be cyclic groups of order two and let $G_z$ be the Klein four-group
$\{ 1,a,b,ab\}$ with $\alpha^e_z, \alpha^f_z$ the inclusions.  Take $H_1 = \{1,x\}$ and $H_0 = \{1,y \}$
also cyclic of order two, with $\psi: x \mapsto y$.  Let $p_e: a \mapsto x, p_f:b \mapsto x$
and $p_z: a \mapsto y, b \mapsto y$.  If we now choose $h_e = x = h_f$ then we have to take
$g_e = a$ and $g_f=b$ as preimages: but then $g_e \alpha^e_z = a \ne b = g_f \alpha^f_z$.  Hence no ordered lifting
homotopy exists.
\end{example}

The distinction between ordered fibrations and ordered morphisms with the ordered homotopy lifting property
was raised but left undecided in \cite{Mthesis}.  We shall follow Steinberg \cite{St} is defining a {\em fibration
of ordered groupoids} to be a star-surjective ordered morphism.  A {\em strong fibration} is defined to be a morphism
of ordered groupoids that has the homotopy lifting property.  We note a couple of constructions for fibrations.

\begin{prop}
\label{cov_is_strong}
\begin{enumerate}
\item[(i)] An ordered covering $p:G \ra H$ is a strong fibration,
\item[(ii)] A pullback of a fibration is a fibration,
\item[(iii)] A pullback of a strong fibration is a strong fibration.
\end{enumerate}
\end{prop}

\begin{proof}
(i) We wish to construct $\widetilde{F}$ in a diagram
\[ \xymatrixcolsep{3pc}
\xymatrix{
A \ar[d]_{i_0} \ar[r]^{f} & G \ar[d]^{p}\\
A \times \I \ar[r]_{F} \ar[ur]_{\widetilde{F}} & H \,.}
\]
For $a \in A$ we must have $(a,0)\widetilde{F}=af$.  Now for any $x \in A_0$ we have
$(x,0)F=xi_0F=xfp$ and since $p$ is a covering, there is a unique $g \in \st_G(xf)$ with $gp = (x,\iota)F$.  We set
$(x, \iota)\widetilde{F}=g$.  Suppose that $y \leq x$ in $A_0$.  Then $yf \leq xf$ and $(yf|g)p \leq gp = (x, \iota)F$ with
$(yf|g)\dom=yf$.  But also $(y,\iota)F \leq (x,\iota)F$ with $(y,\iota)F\dom = yf$.  By uniqueness of restriction we have
$(y,\iota)\widetilde{F} = (yf|g) \leq g$.  We have now defined $\widetilde{F}$ consistently on $A \times \{0\}$ and on
$A_0 \times \I$: these definitions combine to give $(a,\iota)\widetilde{F}=(af)g$ where $g$ is the unique arrow in $\st_G((a\ran)f)$
with $gp=(a \ran,\iota)F$.  This defines an ordered lifting $\widetilde{F}$ of $F$.

(ii) and (iii)  The pullback results are straightforward.
\end{proof}

\begin{prop}
\label{fib_and_immer}
Suppose that a strong fibration $p: G \ra H$ factors as the composite of a morphism $\pi : G \ra T$ and an immersion $\psi: T \ra H$.
Then $\pi$ is a strong fibration.
\end{prop}

\begin{proof}
Consider a commutative square
\[ \xymatrixcolsep{3pc}
\xymatrix{
A \ar[d]_{i_0} \ar[r]^{f} & G \ar[d]^{\pi}\\
A \times \I \ar[r]_{F}  & T \,.}
\]
Extend this diagram by $\psi$ and use the homotopy lifting property of $p=\pi \psi$ to lift $F_\ast = F \psi$ to $G$:
\[ \xymatrixcolsep{4pc}
\xymatrix{
A \ar[d]_{i_0} \ar[r]^{f} & G \ar[d]^{\pi}\\
A \times \I \ar[r]_{F} \ar[dr]_{F \psi} \ar[ur]_{\widetilde{F_\ast}} & T \ar[d]^{\psi} \\
& H \,.}
\]
Then $i_0 \widetilde{F_\ast} = f$.   Consider $(a,\iota) \in A \times \I$.  Then
\[ ((a,\iota) \widetilde{F_\ast} \pi) \dom = (a \dom,0) \widetilde{F_\ast} \pi = a \dom f \pi = (a \dom,0)F = ((a,\iota)\dom)F = ((a, \iota)F)\dom.\]
Now $(a,\iota) \widetilde{F_\ast} \pi \psi = (a,\iota) F \psi$, and since $\psi$ is an immersion, $(a,\iota) \widetilde{F_\ast} \pi = (a,\iota)F$.
Hence $\widetilde{F_\ast}$ lifts $F$.
\end{proof}

\begin{prop}
\label{exp_is_strong}
Let $p:G \ra H$ be a strong fibration of ordered groupoids.  Then for any ordered groupoid $T$, the functor $p_*: \ogpd(T,G) \ra \ogpd(T,H)$ induced by
composition with $p$ is a strong fibration.
\end{prop} 

\begin{proof}
Suppose we are given a commutative square
\begin{equation}
\label{exp_law_fib}
 \xymatrixcolsep{3pc}
\xymatrix{
A \ar[d]_{i_0} \ar[r]^(.35){f} &  \ogpd(T,G) \ar[d]^{p_*}\\
A \times \I \ar[r]_(.4){F}  & \ogpd(T,H) \,.}
\end{equation}
The natural equivalence \eqref{exp_law} implies that there exists a corresponding commutative diagram
\[ \xymatrixcolsep{3pc}
\xymatrix{
T \times A \ar[d]_{i_0} \ar[r]^{f} & G \ar[d]^{p}\\
T \times A \times \I \ar[r]_(.65){F_\natural} \ar[ur]_(.7){\widetilde{F_\natural}} & H}
\]
in which $F_\natural$ lifts to $\widetilde{F_\natural}$.  The functor $\widetilde{F_\natural}$ then corresponds to
$\widetilde{F} : A \times I \ra \ogpd(T,G)$ making \eqref{exp_law_fib} commute.
\end{proof}

\begin{example}
If $p : G \ra H$ is star-surjective, then $p_*: \ogpd(T,G) \ra \ogpd(T,H)$ need not be star-surjective.  We take $E,G,H$ as in Example
\ref{not_ord_fib},  $T=E$, and the inclusion $i: E \ra G$ in $\ogpd(E,G)_0$.  A natural transformation from the functor $ip \in \ogpd(E,H)_0$
is an ordered mapping $\tau: E \ra H$ such that, for all $v \in E$ we have $v \tau \in H_{vp}$.  We take $e \tau = x = f \tau$ and $z \tau = y$.
Then a $\tilde{\tau} \in \ogpd(E,G)$ starting at $i$ with $\tilde{\tau}p=\tau$ must satisfy $v \tilde{\tau} \in G_v$: hence $e \tilde{\tau} = a$, $f \tilde{\tau} = b$,
and $z \tilde{\tau} \ne 1$.  But as before we cannot choose $z  \tilde{\tau} \in \{ a,b,ab \}$ to obtain $z \tilde{\tau} \leq a$ and $z \tilde{\tau} \leq b$.
\end{example}

In the unordered setting, as noted in section \ref{gpdogpd}, any groupoid morphism $\theta : G \ra H$ factorises as the composition of a canonical
fibration $G \ra G/ \ker \theta$ followed by a star-injective morphism $G / \ker \theta \ra H$.  However, if
$\theta : G \ra H$ is ordered, the quotient groupoid $G/ \ker \theta$ may not admit an ordering so that we get
an ordered fibration $G \ra G/ \ker \theta$. 

\begin{example}
\label{bicyclic}
Consider the simplicial groupoid $B = \N \times \N$ in which $\N$ carries the reverse of  its usual total order.
We order arrows in $B$ by $(p+m,p+n) \leq (m,n)$ for all $p,m,n \in \N$: hence $(0,0)$ is a maximum element.  
As a poset, $\N$ is now an infinite descending chain, and so $B$ is inductive.  The inverse semigroup corresponding to $B$ is the \emph{bicyclic
monoid} (see \cite[section3.4]{LwBook}). We define $\theta: B \ra \Z_2$ by $(m,n) \mapsto [m-n]_2$, so that
$\ker \theta = \{ (m,n) : m-n \; \text{is even} \;\}$.  Then $B/ \ker \theta$ is isomorphic to the groupoid $\I$ which can only
carry the trivial ordering, and hence $B \ra \I$ is not an ordered functor.
\end{example}

In the next section we shall present a modified quotient construction for ordered
groupoids that leads to an ordered factorisation of an ordered functor $\theta : G \ra H$.

\section{Quotients of ordered groupoids}
\label{quotog}
\begin{defn}
A subgroupoid $A$ of an ordered groupoid $G$ is an \emph{normal ordered subgroupoid} if
\begin{itemize}
\item $A$ is wide in $G$,
\item if $a \in A$ and $e \leq aa^{-1}$ then the restriction $(e|a) \in A$,
\item if $a \in A$ and $h,k \in G$ satisfy 
\begin{itemize}
\item[$\diamond$] $h \leq g$ and $k \leq g$ for some $g \in G$ (that is, $h$ and $k$ have
an upper bound in $G$),
\item[$\diamond$] $h^{-1}ak$ is defined in $G$,
\end{itemize}
then $h^{-1}ak \in A$.
\end{itemize}
\end{defn}

In the unordered case (when $\leq$ is equality) we recover Higgins' definition \cite{HiBook} of a normal subgroupoid: the second
condition is then vacuous, and in the third condition we have $h=k$ and then $h^{-1}ak$ being defined forces $a$ to be
in a local subgroup of $G$.  In the ordered case, if $A$ is a disjoint union of groups, then $h^{-1}ak$ being
defined implies that $hh^{-1}=kk^{-1}$, and then each of $h,k$ is the restriction of $g$, so by uniqueness of 
restriction, $h=k$ again and we recover Matthews' definition \cite{Mthesis}.

\begin{lemma}
\label{ker_is_normal}
The kernel of an ordered morphism $\theta: G \ra H$ is a normal ordered subgroupoid of $G$.  
\end{lemma}

\begin{proof}Clearly $G_0 \subseteq \ker \theta$,
so suppose that $a \in \ker \theta$ , $h,k \in G$ with the composition $h^{-1}ak$ defined in $G$, and that there exists $g \in G$ with $h \leq g$ and $k \leq g$. Suppose that $a \theta = z \in H_0$. Then $(h^{-1}\theta)(k \theta)$ is defined in $H$, and $h \theta \leq g \theta, k \theta \leq g \theta$.  It follows that $h \theta = (z|g \theta) = k \theta$ and hence
\[ (h^{-1}ak)\theta = (h^{-1}\theta)(k \theta) = (h\theta)^{-1}(k \theta) \in H_0 \]
and so $h^{-1}ak \in \ker \theta$.
\end{proof}

\begin{example}
An inverse subsemigroup $K$ of an inverse semigroup $S$ is normal (as in \cite{LwBook})  if and only if $G(K)$ is a normal ordered subgroupoid
of $G(S)$. For suppose that $G(K)$ is normal in $G(S)$.  Consider $s^{-1}ks$ with $k \in K$ and $s \in S$.  We can write, in $S$:
\[ s^{-1}ks = (s^{-1}kss^{-1}k^{-1}) (ss^{-1}kss^{-1}) (k^{-1}ss^{-1}ks) \]
and the three bracketed terms on the right-hand side are now three composable arrows in $G(S)$.  The middle term is
$\leq k$ and so is in $G(K)$.   The first term is $\leq s^{-1}$ and the third term is $\leq s$ and so the whole product is in
$G(K)$ and so, as an element of $S$, is in $K$. On the other hand, if $K$ is a normal inverse subsemigroup of $S$, then
consider a composition $s^{-1}kt$ defined in $G(K)$ and where $s \leq u, t \leq u$ for some $u \in G$.  Then in $K$
we have $s^{-1}kt \leq u^{-1}ku$ and $u^{-1}ku \in K$ by normality.  So the composed arrow $u^{-1}ku \in G(K)$  and hence
so is $s^{-1}kt$ since $s^{-1}kt \leq u^{-1}ku$.
\end{example}

\begin{lemma}
\label{def_g_over_a}
Let $A$ be a normal ordered subgroupoid of the ordered groupoid $G$.  Then the relation
\[ g \simeq_A h \iff \; \text{there exist} \; a,b,c,d \in A \; \text{such that} \; agb \leq h \; \text{and} \; chd \leq g \]
is an equivalence relation on $G$.  The relation
\[ [g] \leq [k] \iff \; \text{there exist} \; a,b \in A \; \text{such that} \; agb \leq k  \]
is a well-defined partial order on the set $G \twobar A$ of equivalence classes of $\simeq_A$.
\end{lemma}

\begin{proof}
It is obvious that $\simeq_A$ is reflexive and symmetric.  Suppose that $g \simeq_A h$ and $h \simeq_A k$.
Then there exist $a,b,u,v \in A$ such that $agb \leq h$ and $uhv \leq k$.  Then $aa^{-1} \leq hh^{-1} = u^{-1}u$
and $b^{-1}b \leq h^{-1}h = vv^{-1}$, and 
$(u|aa^{-1})agb(b^{-1}b|v) \leq uhv \leq k$ with $(u|aa^{-1})a, b(b^{-1}b|v) \in A$.  Similarly there exist $p,q \in A$
with $pkq \leq g$ and $g \simeq_A k$.

We now show that $\leq$ is well-defined.  Suppose that $g \simeq_A g'$ and that $k \simeq_A k'$. Hence there
exist $p,q,u,v \in A$ with $pg'q \leq g, ukv \leq k'$.  Now if in addition
there exist $a,b \in A$ with $agb \leq k$ then $(u|aa^{-1})agb(b^{-1}b|v) \leq ukv \leq k'$ and
\[ ((u|aa^{-1})a|pp^{-1})pg'q(q^{-1}q|b(b^{-1}b|v)) \leq (u|aa^{-1})agb(b^{-1}bv).\]
Hence $g' \leq k'$.

Now it is clear that $\leq$ is reflexive and transitive, and by definition of $\simeq_A$ it is anti-symmetric. 
\end{proof} 

\begin{lemma}
\label{obj_g_over_a}
Let $e,f \in G_0$.  Then $e \simeq_A f$ if and only if there exist $a,p \in A$ such that
\[ aa^{-1} \leq e, a^{-1}a = f \; \text{and} \; pp^{-1} \leq f, p^{-1}p = e \,. \]
\end{lemma}

A pair $(a,p)$ of arrows of $A$ realizing the relation $\simeq_A$ for $e,f \in G_0$ is called an $A$--{\em nexus}
between $e$ and $f$.

There is a natural ordered groupoid structure on $G \twobar A$ that we now begin to describe.

\begin{lemma}
Suppose that $g \simeq_A h$.  Then $g^{-1} \simeq_A h^{-1}$ and $gg^{-1} \simeq_A hh^{-1}$.
\end{lemma}

\begin{proof}  Suppose that there exist $a,b,u,v \in A$ such that $agb \leq h$ and $uhv \leq g$.  Then
$b^{-1}g^{-1}a^{-1} \leq h^{-1}$ and $v^{-1}h^{-1}u^{-1} \leq g^{-1}$.  Furthermore,
$(u,a)$ is an $A$--nexus between $gg^{-1}$ and $hh^{-1}$.
\end{proof}

When we represent relationships in $G$ diagrammatically, we shall indicate the partial order by a dashed line with the order
decreasing down the page.

\newpage

\begin{lemma}
\label{nexus_comp}
Let $g,h \in G$ with $g^{-1}g \simeq_A hh^{-1}$.  Let $(a,p)$ be an $A$--nexus between $g^{-1}g$ and $hh^{-1}$.
Define
\begin{multicols}{2}
\begin{align*}
g' &= (g|aa^{-1}) \\
h' &= (pp^{-1}|h) \\
a' &= (a |pp^{-1}) \\
p' &= (p|aa^{-1}) \\
g'' &= (g| a'a'\,^{-1}) \\
h'' &= (p'p'\,^{-1}|h).
\end{align*}

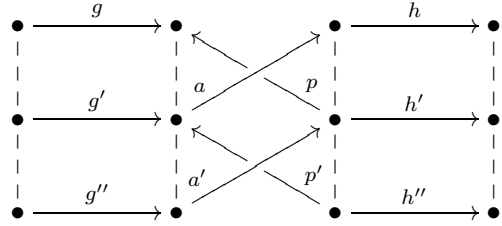
\columnbreak

\[  \xymatrixcolsep{4pc}
\xymatrix{
\bullet \ar[r]^g \ar@{--}[d] & \bullet \ar@{--}[d] & \bullet \ar[r]^h \ar@{--}[d] & \bullet \ar@{--}[d]\\
\bullet \ar[r]^{g'} \ar@{--}[d] & \bullet \ar[ur]^(.2){a} \ar@{--}[d] & \bullet \ar[r]^{h'} \ar@{--}[d] \ar[ul]_(.2){p} |\hole & \bullet \ar@{--}[d]\\
\bullet \ar[r]^{g''}  & \bullet \ar[ur]^(.2){a'} & \bullet \ar[r]^{h''} \ar[ul]_(.2){p'} |\hole & \bullet \\
} \]
\end{multicols}
Then $g'ah \simeq_A gp^{-1}h'$, and if we make an alternative choice of nexus $(a_1,p_1)$ leading to elements
$g'_1,h'_1$ then
\[ g'_1a_1h \simeq_A g'ah \simeq_A gp^{-1}h' = gp_1^{-1}h'_1 \,. \] 
\end{lemma}

\begin{proof}
We have $g'ah \geq g''a'h' = g''a'pg^{-1}gp^{-1}h'$ with $g''a'pg^{-1} \in A$ by normality of $A$, and similarly
$gp^{-1}h' \geq g'p'\,^{-1}h'' = g'ahh^{-1}a^{-1}p'\,^{-1}h''$ with $h^{-1}a^{-1}p'\,^{-1}h'' \in A$.  

For the alternative choice of nexus we have $a^{-1}a=a_1^{-1}a_1=hh^{-1}$ and
$g'_1a_1h \simeq_A g'_1a_1a^{-1}g'~^{-1}g'ah$ with $g'_1a_1a^{-1}g'\,^{-1} \in A$.  Hence $g'_1a_1h \simeq_A g'ah$
and similarly $gp^{-1}h' \simeq_A gp_1^{-1}h'_1$.
\end{proof}

Given $g,h \in G$ with $g^{-1}g \simeq_A hh^{-1}$ we have now defined a unique class in $G \twobar A$ in terms of an 
$A$--nexus between $g^{-1}g$ and $hh^{-1}$, but independent of the $A$--nexus chosen.  We temporarily denote this class
by $g \between h \in G\twobar A$.  Hence in the notation of Lemma \ref{nexus_comp},
\[ g \between h = [g'ah] = [gp^{-1}h']. \]

\begin{lemma}
\label{nexus_comp_wd}
Suppose that $g \simeq_A g_1$ and that $h \simeq_A h_1$.  Then $g \between h = g_1 \between h_1$ in $G \twobar A$.
\end{lemma}

\begin{proof}
Fix $h$ and choose an $A$--nexus $(a_1,p_1)$ between $g_1^{-1}g_1$ and $hh^{-1}$.  Let $g'_1 = (g_1|a_1a_1^{-1})$
and $h'_1 = (p_1p_1^{-1}|h)$.  Then
$g_1 \between h = [g'_1a_1h]=[g_1p_1^{-1}h'_1] \in G \twobar A$. There exist $u,v \in A$ with $ug_1v \leq g$.
Let $u' = (u|g'_1g'_1 \,^{-1})$ and $v' = (a_1a_1^{-1}|v)$.  
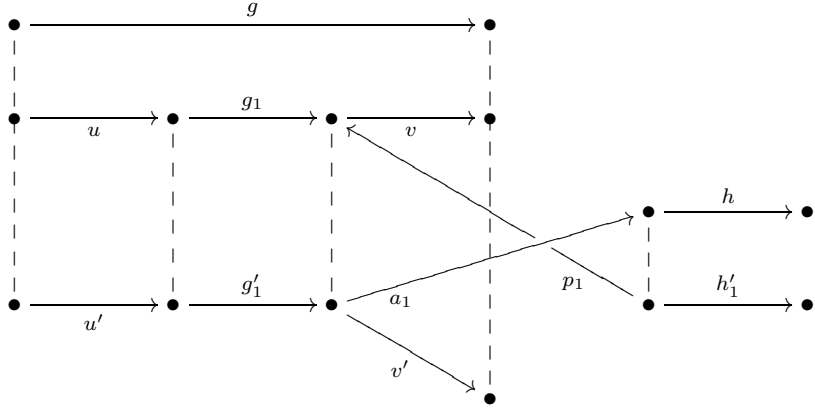
\[ \xymatrixcolsep{4pc}
\xymatrix{
\bullet \ar[rrr]^{g} \ar@{--}[ddd] &&& \bullet \ar@{--}[dddd] && \\
\bullet \ar[r]_{u} & \bullet \ar[r]^{g_1} \ar@{--}[dd] & \bullet \ar[r]_{v} \ar@{--}[dd] & \bullet & & \\
&&&& \bullet \ar[r]^{h} \ar@{--}[d] & \bullet \\
\bullet \ar[r]_{u'} & \bullet \ar[r]^{g'_1} & \bullet \ar[rd]_{v'} \ar[rru]_(.2){a_1} & & \bullet \ar[r]^{h'_1} \ar[lluu]^(.2){p_1} |!{[ll];[u]}\hole & \bullet \\
&&& \bullet &&
}
\]
Now $v'\,^{-1}a_1$ is a component of an $A$--nexus
between $g^{-1}g$ and $hh^{-1}$ and so can be used to define $g \between h$.  By uniqueness of restriction we have
$(g|v'\,^{-1}v') = u'g'_1v'$ and so 
\begin{align*}
g \between h &= [ u'g'_1v'v'\,^{-1}a_1h] \\
&= [u'g'_1a_1h] \\
&= [g'_1a_1h] = g_1 \between h.
\end{align*}
Similarly $g_1 \between h = g_1 \between h_1$.
\end{proof}

We now have a well-defined composition of $\simeq_A$--classes:
$[g][h] = g \between h$.  Suppose that $g^{-1}g \simeq_A hh^{-1}$ and $h^{-1}h \simeq_A kk^{-1}$.  Choose
$A$--nexuses $(a,p)$ between $g^{-1}g$ and $hh^{-1}$ and $(b,q)$ between $h^{-1}h$ and $kk^{-1}$.  Then
$g \between h = [g'ah]$ and $g'ah \between k = [g'ahq^{-1}k']$.  But 
$h \between k = [hq^{-1}k']$ and $g \between hq^{-1}k' = [g'ahq^{-1}k']$, and so the composition is associative.
This last observation then establishes the following result.

\begin{prop}
\label{quotisgpd}
$G \twobar A$ is a groupoid under the operation of composition of $\simeq_A$--classes.
\end{prop}

\begin{lemma}
\label{nexus_lma}
Let $h \in G$ and suppose that $(a,p)$ is an $A$--nexus between $e \in G_0$ and $h \dom$.  Then for  $h' = (p \dom|h)$,
we have $[h']=[h]$ in $G \twobar A$.
\end{lemma}

\begin{proof}
Since $h' \leq h$ in $G$ we have $[h'] \leq [h]$.  Now set $p'=(p|a \dom)$ and $h''=(p' \dom |h)$.  Then
\[ h \simeq_A h(h^{-1}a^{-1}p'^{-1}h'') = a^{-1}p'^{-1}h'' \simeq_A h'' \leq h' \]
and so $[h] \leq [h']$.
\end{proof}

We shall now proceed to verify that the partial order on $\simeq_A$--classes given in Lemma \ref{def_g_over_a} makes
$G \twobar A$ into an ordered groupoid.

\begin{lemma}
\label{og1}
If $[g] \leq [h]$ in $G \twobar A$ then $[g]^{-1}=[g^{-1}] \leq [h^{-1}]=[h]^{-1}$.
\end{lemma}

\begin{proof}
We have $agb \leq h$ for some $a,b \in A$ and so $b^{-1}g^{-1}a^{-1} \leq h^{-1}$.
\end{proof}

\begin{lemma}
\label{og2}
Suppose that $[g_1] \leq [h_1]$ and that $[g_2] \leq [h_2]$ in $G \twobar A$ and that the compositions $[g_1][g_2]$ and
$[h_1][h_2]$ exist.  Then $[g_1][g_2] \leq [h_1][h_2]$.
\end{lemma}

\begin{proof}
There exist $a_1,b_1,a_2,b_2 \in A$ such that $a_1g_1b_1 \leq h_1$ and $a_2g_2b_2 \leq h_2$.  Since $[g_i]=[a_ig_ib_i]$ we
may as well assume that $g_i \leq h_i$ for $i=1,2$.  We have $A$--nexuses $(a,p)$ between $h_1^{-1}h_1$ and $h_2h_2^{-1}$
and $(b,q)$ between $g_1^{-1}g_1$ and $g_2g_2^{-1}$.  We set $p' = (p|g_1^{-1}g_1)$ and
$h''_2 = (p'p' \,^{-1}|h'_2)$.  Then
\[ g_1q^{-1}g'_2  \simeq_A g_1q^{-1}g'_2g'_2\,^{-1}qp'\,^{-1}h''_2 = g_1p'h''_2 \]
since $g'_2,h''_2 \leq h_2$ and so $g'_2\,^{-1}qp'\,^{-1}h''_2 \in A$, and
\[ g_1p'\,^{-1}h''_2 \leq h_1p^{-1}h'_2.\]
Therefore $[g_1][g_2] = [g_1q^{-1}g'_2] \leq [h_1p^{-1}h'_2] = [h_1][h_2]$.
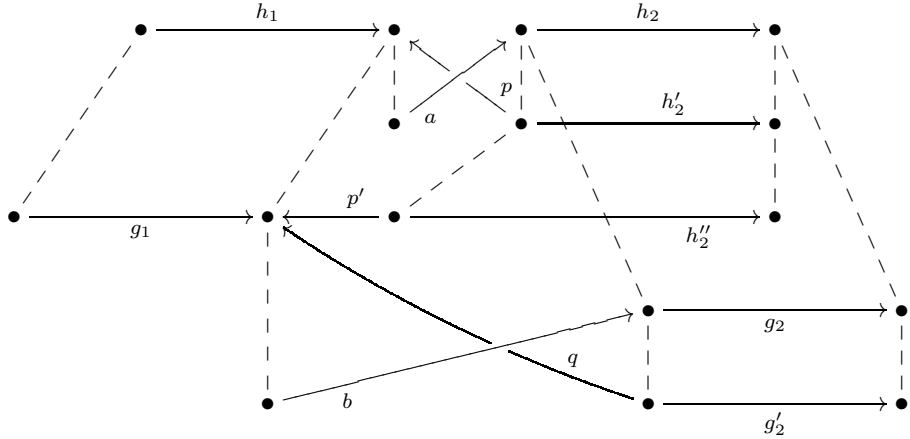
\[ \xymatrixcolsep{3pc}
\xymatrix{
& \bullet \ar[rr]^{h_1} \ar@{--}[ddl] && \bullet  \ar@{--}[ddl] \ar@{--}[d] & \bullet \ar[rr]^{h_2} \ar@{--}[d] \ar@{--}[dddr] && \bullet \ar@{--}[d] \ar@{--}[dddr] & \\
&  && \bullet \ar[ur]_(.2){a} & \bullet \ar[rr]^(.6){h'_2} \ar[ul]_(.2){p} |\hole \ar@{--}[dl] && \bullet \ar@{--}[d] & \\
\bullet \ar[rr]_{g_1} &  & \bullet \ar@{--}[dd] & \bullet \ar[l]_(.3){p'} \ar[rrr]_(.8){h''_2} &&& \bullet & \\
&&&&& \bullet \ar[rr]_{g_2} \ar@{--}[d]&& \bullet \ar@{--}[d] \\
&& \bullet \ar[urrr]_(.2){b} &&& \bullet \ar[rr]_{g'_2} \ar@/^/[llluu]_(.2){q} |!{[lll];[u]}\hole && \bullet 
}
\]
\end{proof}

\begin{lemma}
\label{og3}
Suppose that $[e] \leq [gg^{-1}]$.  Then there exists a unique arrow $[k] \in G \twobar A$ such that $[k] \leq [g]$ and
$[k][k]^{-1}=[e]$.
\end{lemma}

\begin{proof}
Since $[e] \leq [gg^{-1}]$ there exists $a \in A$ with $aa^{-1} \leq gg^{-1}$ and $a^{-1}a=e$.  Let $g' = (aa^{-1}|g)$.
Then $[g'] \leq [g]$ and $[g'][g']^{-1} = [aa^{-1}]=[e]$.

Now suppose that $[k] \leq [g]$ and $[k][k]^{-1}=[e]$.  We shall show that $[k]=[g']$. There exist $u,v \in A$ such that
$ukv \leq g$ and an $A$--nexus $(b,q)$ between $kk^{-1}$ and $e$, so that
\[ bb^{-1} \leq kk^{-1} \,, b^{-1}b=e \,, qq^{-1} \leq e \,, q^{-1}q=kk^{-1} \,. \]
Let $k' = (bb^{-1}|k), u' = (u|bb^{-1})$ and $v' = (k'\,^{-1}k'|v)$.
Then $u'k'v' \leq ukv \leq g$ and 
\begin{align*}
g' & \simeq_A g' g'\,^{-1}ab^{-1}u'\, ^{-1}u'k'v' \\
\intertext{since $g' \leq g, u'k'v' \leq g$, and $ab^{-1}u' \,^{-1} \in A$,}
&= ab^{-1}k'v' \simeq_A k' \leq k.
\end{align*}
Let $a' = (a|qq^{-1})$ and $g'' = (a'a' \,^{-1}|g')$. 
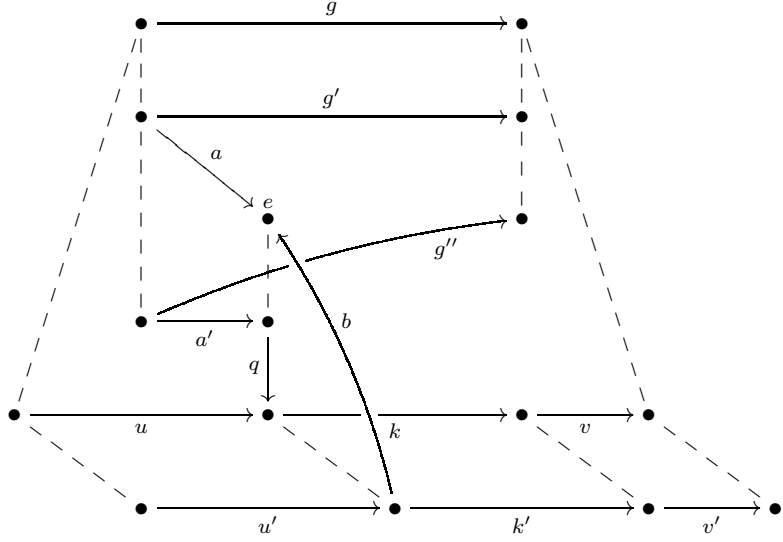
\[ \xymatrixcolsep{3pc}
\xymatrix{
& \bullet \ar[rrr]^{g} \ar@{--}[ddddl] \ar@{--}[d] &&& \bullet \ar@{--}[d] \ar@{--}[ddddr] && \\
& \bullet \ar[rrr]^{g'} \ar[dr]^{a} \ar@{--}[dd] &&& \bullet \ar@{--}[d] && \\
&& \stackrel{e}{\bullet} \ar@{--}[d] && \bullet && \\
& \bullet \ar[r]_{a'} \ar@/^/[urrr]_(.8){g''} |(.42)\hole & \bullet \ar[d]_{q} &&&& \\
\bullet \ar[rr]_{u} \ar@{--}[dr] & & \bullet \ar[rr]_{k} |(.4)\hole \ar@{--}[dr] & & \bullet \ar[r]_{v} \ar@{--}[dr] & \bullet \ar@{--}[dr] & \\ 
& \bullet \ar[rr]_{u'}  & & \bullet \ar[rr]_{k'} \ar@/_/[uuul]_(.6){b} & & \bullet \ar[r]_{v'} & \bullet } \]
Then
\begin{align*}
k & \simeq_A kvv^{-1}k^{-1}u^{-1}uq^{-1}a'\,^{-1}g''\\
\intertext{since $ukv \leq g, g'' \leq g$ and $v, ug^{-1}a'\,^{-1} \in A$,}
&= q^{-1}a,\,^{-1}g'' \simeq g'' \leq g'.
\end{align*}
Hence $[k] \leq [g']$ and by symmetry $[g'] \leq [k]$.
\end{proof}

\begin{theorem}
\label{quot_is_ogpd}
$G \twobar A$ is an ordered groupoid and the quotient map $\varpi: G \ra G \twobar A$ is a fibration.
\end{theorem}

\begin{proof}
The ordered groupoid structure follows from  Proposition \ref{quotisgpd} and Lemmas \ref{og1}, \ref{og2}, and \ref{og3}.
To show that $\varpi: g \mapsto [g]$ is a fibration, consider the restriction $\st_G(e) \ra \st_{G \twobar A}[e]$ and
$[h] \in \st_{G \twobar A}[e]$.  There exists an $A$--nexus $(a,p)$ between $e$ and $hh^{-1}$.  Defining $h',h'',a',p'$
as in Lemma \ref{nexus_comp},
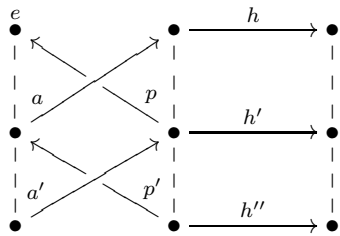
\[  \xymatrixcolsep{4pc}
\xymatrix{
\stackrel{e}{\bullet} \ar@<-.05ex>@{--}[d] & \bullet \ar[r]^h \ar@{--}[d] & \bullet \ar@{--}[d]\\
\bullet \ar[ur]^(.2){a} \ar@{--}[d] & \bullet \ar[r]^{h'} \ar@{--}[d] \ar[ul]_(.2){p} |\hole & \bullet \ar@{--}[d]\\
\bullet \ar[ur]^(.2){a'} & \bullet \ar[r]^{h''} \ar[ul]_(.2){p'} |\hole & \bullet \\
} \]
we have $[p^{-1}h']=[h']$ , and $[h']=[h]$ by Lemma \ref{nexus_lma}.   Hence we have $p^{-1}h' \in \st_G(e)$ and 
\[ \varpi: p^{-1}h' \mapsto [p^{-1}h']=[h']=[h] .\]
\end{proof}

\medskip
Example \ref{not_ord_fib}, coupled with Example \ref{quot_egs}(iv) below shows that $\varpi$ need not be a strong fibration.

If we begin with an ordered functor $\theta: G \ra H$, then by Lemma \ref{ker_is_normal}, its kernel $\ker \theta$ is a normal ordered subgroupoid of $G$ and we can form the quotient ordered groupoid $G \twobar \ker \theta$.  We show that $\theta$ factors through
$G \twobar \ker \theta$ as the composition of the fibration $\varpi$ and a star-injective functor.

\begin{theorem}
\label{factor_as_fib_imm}
An ordered functor $\theta: G \ra H$ induces a star-injective functor  $\psi: G \twobar \ker \theta \ra H$ such that $\theta = \varpi \psi$.
\end{theorem} 

\begin{proof}
We define $\psi : G \twobar \ker \theta \ra H$ by $\psi: [g] \mapsto g \theta$. It is clear that $\psi$ is an ordered functor.
Suppose that $[k] \in \st_{G \twobar \ker \theta} [g]$ and that $k \theta = g \theta$.  Let $(a,p)$ be a $(\ker \theta)$--nexus between
$gg^{-1}$ and $kk^{-1}$.
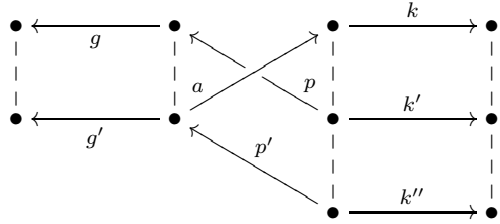
\[  \xymatrixcolsep{4pc}
\xymatrix{
\bullet  \ar@{--}[d] & \bullet \ar[l]^g \ar@{--}[d] & \bullet \ar[r]^{k} \ar@{--}[d] & \bullet \ar@{--}[d]\\
\bullet  & \bullet  \ar[l]^{g'} \ar[ur]^(.2){a}  & \bullet \ar[r]^{k'} \ar@{--}[d] \ar[ul]_(.2){p} |\hole  & \bullet \ar@{--}[d]\\
&   & \bullet \ar[r]^{k''} \ar[ul]_{p'}  & \bullet \\
} \]
Now 
\[(g^{-1}p^{-1}k')\theta = (g\theta)^{-1}(p \theta)^{-1}k' \theta = (g\theta)^{-1}k' \theta \leq (g\theta)^{-1}k \theta \in H_0.\]
It follows that $g^{-1}p^{-1}k' \in \ker \theta$, and so
\[ k' = pg(g^{-1}p^{-1}k') \simeq_{\ker \theta} g. \]
But, by Lemma \ref{nexus_lma}, we have $[k]=[k']$ and so $[k]=[g]$.
\end{proof}

\begin{cor}
\label{fib_as_fib_cov}
If $\theta: G \ra H$ is an ordered fibration then $\psi: G \twobar \ker \theta \ra H$ is a covering.
\end{cor}

\begin{proof}
In the triangle
\[  
\xymatrix{
\st_G(e) \ar[r]^(.4){\varpi} \ar[d]_{\theta} & \st_{G \twobar \ker \theta} [e] \ar[dl]^{\psi} \\
\st_H(e \theta) &
} \]
$\psi$ is star-injective by Theorem \ref{factor_as_fib_imm}, whilst $\theta$ and $\varpi$ are star-surjective by assumption and by
Theorem \ref{quot_is_ogpd} respectively.  It follows that $\psi$ is star-surjective, and so $\psi$ is a covering.
\end{proof}

\begin{example}
\label{quot_egs}
(i) For a groupoid $G$, the quotient $G/G$ is isomorphic to the set $\pi_0(G)$ of connected components of $G$, regarded as a trivial groupoid.
If $G$ is ordered, then its ordering induces a preorder on $\pi_0(G)$, defined as follows.  If $\lfloor g \rfloor$ denotes the connected component of
$g \in G$, then $\lfloor g \rfloor \leq \lfloor h \rfloor$ if and only if for each $h' \in \lfloor h \rfloor$ there exists $g' \in \lfloor g \rfloor$ with $g' \leq h'$.
As a preordered set, $\pi_0(G)$ has a canonical partially ordered quotient $Q(G)$, and the poset $Q(G)$ is isomorphic to the ordered quotient $G \twobar G$.

\medskip
(ii) We return to Example \ref{bicyclic}.  For any $p \in \N$ and any even $a \geq p$ the pair $(a,0),(p,p)$ is a $(\ker \theta)$--nexus between $p$ and $0$
and hence $B \twobar \ker \theta$ has one object and is a group.  Moreover, if $(m,n) \not\in \ker \theta$ then $m,n$ have opposite parity and then
$(m,n) \simeq_{\ker \theta} (0,1)$, since
$(m,m)(m,n)(n,m+1) = (m,m+1) \leq (0,1)$ and $(2m,0)(0,1)(1,m+n) = (2m,m+n) \leq (m,n)$.  Hence  $B \twobar \ker \theta$ is isomorphic to $\Z_2$.

\medskip
(iii) Let $\Delta^n$ be the {\em simplicial groupoid} on the set $I_n=\{0,1,\dotsc,n\}$ (see \cite{HiBook}): hence
$(\Delta^n)_0=I_n$ and there is a unique arrow $(i,j)$ from $i$ to $j$.  $\Delta^n$ is freely generated by the
obvious directed chain $\Gamma_n$ with vertex set $I_n$.  We set $\Delta = \bigsqcup_{n \geq 1} \Delta^n$.  The natural partial order is trivial between 
elements of each $\Delta^n$.  If $\gamma \in \Delta^p$ and $\delta \in \Delta^q$ with $p<q$ then $\gamma \geq \delta$ if and only if $\delta$ is the image of 
$\gamma$ under an embedding $\Gamma_p \hookrightarrow \Gamma_q$.   $\Delta$ is the ordered groupoid equivalent to the monogenic free inverse semigroup
$\fis(x)$.

Adjoin a minimum idempotent $z$ to $\I$: we can then define a functor $\theta : \Delta \ra \I \cup \{ z \}$ by mapping $\Delta^1$ to $\I$ via an isomorphism and, for all $n>1$,
mapping $\delta \in \Delta^n$ to $z$.  Then $\Delta \twobar \ker \theta$ is the ordered groupoid $\Delta^1 \cup \{ e_i : i \geq 2 \}$ in which the $e_i$ form an infinite
descending chain.

\medskip
(iv) For a presheaf of groups $G=({\mathscr G},P)$ as in Example \ref{presheaf}, an ordered normal subgroupoid is a presheaf of groups $N =({\mathscr N},P)$ where,
for each $x \in P$, the group $N_x$ is a normal subgroup of $G_x$, and the linking maps are just the restrictions of those in $G$ so that, whenever $x \geq y$ in $P$, 
then $N_y \supseteq N_x \alpha^x_y$.  Then $g \simeq_N h$ if and only if $g,h \in G_x$ for some $x$ and $g^{-1}h \in N_x$.  Hence $G \twobar N$ is isomorphic to 
the presheaf of groups $Q = ({\mathscr Q},P)$ where $Q_x = G_x/N_x$ and the linking maps for $Q$ are induced by those in $G$.

\medskip
(v) Fix a group $G$ and an endomorphism $\alpha: G \ra G$.  The \emph{Bruck-Reilly groupoid} $\br(G,\alpha)$ (corresponding to the Bruck-Reilly extension of inverse 
semigroups introduced in \cite{Rei})  determined by $(G,\alpha)$ is the ordered
groupoid with object set $\N$, ordered by the reverse of the natural total order $\geq$ on $\N$, and arrow set $\N \times G \times \N$.  For an arrow $(m,g,n)$ we have
$(m,g,n)\dom=m, (m,g,n)\ran=n$ and the composition of arrows is $(m,g,n)(n,h,q)=(m,gh,q)$.  The ordering $\succcurlyeq$ on arrows is defined as follows:
\[ (m,g,n) \succcurlyeq (p,h,q) \; \text{if and only if} \; m \leq p, m-n=p-q, \; \text{and} \; h=g \alpha^{p-m} \,.\]
An ordered functor  $\Theta: \br(G,\alpha) \ra \br(H,\beta)$ that is the identity on objects is determined by a group homomorphism $\theta: G \ra H$
satisfying $\alpha \theta = \theta \beta$, and then $\Theta: (m,g,n) \mapsto (m,g \theta,n)$.  Hence
\[ \ker \Theta = \{ (m,g,m) : m \in \N, g \in \ker  \theta \}.\]
Now if $g \in \ker \theta$ then
$g \alpha \theta = g \theta \beta = e_H$ and so $\alpha: \ker \theta \ra \ker \theta$, and it follows that $\ker  \Theta = \br(\ker \theta,\alpha)$.

It is easy to see that  $(m,g,n) \simeq_{\ker \Theta} (p,h,q)$ if and only if  $m=p, n=q$ and $g \theta = h \theta$ and so we can identify
$\br(G,\alpha) \twobar \ker \theta$ as $\br(G/ \ker \theta, \ol{\alpha})$, where $\ol{\alpha}$ is induced by $\alpha$.

\medskip
(vi) If $A$ is a normal inverse subsemigroup of an inverse semigroup $S$ then $S \twobar A$ need not be an inductive groupoid.  Consider $S$ as shown below:
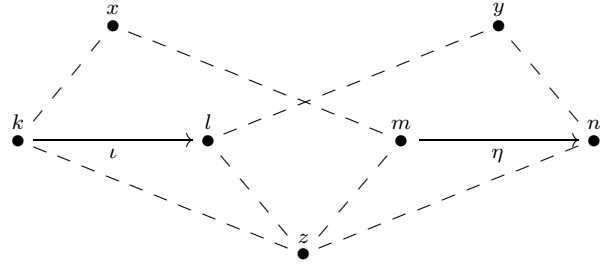
\[  
\xymatrix{
& \stackrel{x}{\bullet} \ar@{--}[ld] \ar@{--}[rrrd] &&&& \stackrel{y}{\bullet} \ar@{--}[llld] \ar@{--}[rd] & \\
\stackrel{k}{\bullet} \ar[rr]_{\iota} \ar@{--}[rrrd] & & \stackrel{l}{\bullet} \ar@{--}[dr]  &&  \stackrel{m}{\bullet} \ar[rr]_{\eta} \ar@{--}[dl] && \stackrel{n}{\bullet} \ar@{--}[llld] \\
&&& \stackrel{z}{\bullet} &&&
} \]
and let $A=S$.  The ordered groupoid $S \twobar S$ is the trivial ordered groupoid
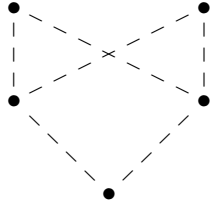
\[  
\xymatrix{
\bullet \ar@{--}[d] \ar@{--}[drr] && \bullet \ar@{--}[d] \ar@{--}[dll]\\
\bullet \ar@{--}[dr] && \bullet \ar@{--}[dl] \\
& \bullet &
} \]
which is not inductive.
\end{example}

\section{The maximum enlargement theorem}
\label{maxenlthm}
Theorem \ref{factor_as_fib_imm} gives a canonical factorization of an ordered groupoid morphism as a fibration followed by a
star-injective functor.  Ehresmann's Maximum Enlargement Theorem \cite{Ehres} then provides a canonical factorization of a star-injective functor as a
well-structured ordered embedding -- called an enlargement -- followed by a covering. For an exposition of the Maximum Enlargement
Theorem and its applications see \cite[chapter 11]{LwBook}.  In this section we give a short account
of the Maximum Enlargement Theorem based solely on the notion of an ordered groupoid acting on a poset: this basis for a proof
of the theorem was first set out in \cite{Gi3} and fully developed in the PhD thesis of the third author \cite{ecm_phd}.

The notions of a groupoid action on a groupoid and the associated construction of the semidirect product seem to have been
first defined in \cite{Br2}.  We describe these notions for ordered groupoids: our definitions are equivalent to those given by
Steinberg \cite{St}.  An {\em action} of an ordered groupoid $G$ on an ordered groupoid $A$ is determined by the following data.  We
are given an ordered functor $\omega: A \ra G_0$ and for each pair $(a,g)$ with $a \omega = g \dom$ there
exists an element $a \lhd g \in A$ such that
\begin{itemize}
\item $(a \lhd g)\omega = g \ran$,
\item if $g,h \in G$ are two composable arrows and if $a \in A$ satisfies $a \omega = g \dom$, then
$a \lhd gh = (a \lhd g) \lhd h$,
\item if $a,b \in A$ are two composable arrows and if $g \in G$ satisfies $a \omega = g \dom = b \omega$,
then $a \lhd g$ and $b \lhd g$ are composable, and $(ab) \lhd g = (a \lhd g)(b \lhd g)$,
\item $a \lhd a \omega = a$,
\item if $a \leq b$ in $A$ and $g \leq h \in G$ with $a \omega = g \dom$ and $b \omega = h \dom$ then
$a \lhd g \leq b \lhd h$ in $A$.
\end{itemize}
This definition includes that of an ordered groupoid $G$ acting on a poset $P$, where $P$ is considered as a trivial groupoid.

Given an action of the ordered groupoid $G$ on the ordered groupoid $A$ the {\em semidirect product}
or {\em action groupoid} $G \ltimes A$ was defined by Brown \cite{Br2} in the unordered case, and by Steinberg 
\cite{St} for ordered groupoids.  $G \ltimes A$ is an ordered groupoid whose  set of arrows is the pullback
$\{ (g,a) : a \omega = g \ran \}$.  The domain and range maps are defined by
\[ (g,a) \dom = (g \dom, (a \lhd g^{-1}) \dom) \; \text{and} \; (g,a)\ran = (g\ran,a\ran) \]
and the composition is
\[ (g,a)(h,b) = (gh,(a \lhd h)b) \]
defined when $g \ran = h \dom$ and $a \ran = (b \lhd h^{-1})\dom$.  The ordering is componentwise.  We note that
we can identify $(G \ltimes A)_0$ with $A_0$ since if $(x,y) \in (G \ltimes A)_0$ then $x = y \omega$.

\begin{lemma}
\label{sdp_fib}
The projection map $\pi: (g,a) \mapsto g$ is a strong fibration $G \ltimes A \ra G$.
\end{lemma}

\begin{proof}
Given a commutative square
\[ \xymatrixcolsep{3pc}
\xymatrix{
B \ar[d]_{i_0} \ar[r]^{f} & G \ltimes A \ar[d]^{\pi}\\
B \times \I \ar[r]_{F} & G}
\]
in which $bf = (g_b,a_b) \in G \ltimes A$ and $(y,\iota)F=h_y \in G$, we define
\[ (b,\iota)\widetilde{F} = (g_bh_{b\ran}, a_b \lhd h_{b \ran})\,.\]
This is an ordered functor lifting $F$.
\end{proof}

\subsection{Actions on posets and coverings of ordered groupoids}
There is an equivalence of categories between the category of coverings of a groupoid $G$ and the category of actions of
$G$ on sets.  This idea originates with Ehresmann \cite{Ehres} but is discussed explicitly in \cite[Proposition 30]{HiBook} and in \cite[Proposition 1.2]{BHK}.
To an action of $G$ on the set $X$ there corresponds the projection $\pi: G \ltimes X \ra G$, and given any covering
$C \ra G$ there is an action of $G$ on $C_0$.  This equivalence generalizes to one for ordered groupoids acting on
posets. The details of this generalisation were set out in \cite{ecm_phd}, and we record the essentials here.

Let $\gamma: C \ra G$ be an ordered covering of groupoids.  We define an action of $G$ on $C_0$ as follows.
The covering $\gamma$ restricts to an ordered mapping $\gamma: C_0 \ra G_0$, and if $x \gamma = g \dom$ there
exists a unique $c \in \st_C(x)$ with $c \gamma = g$: we define $a \lhd g = c \ran$.  The projection map
$\pi: G \ltimes C_0 \ra G$ is then a covering naturally isomorphic to $C \ra G$.  On the other hand, given an action of
$G$ on a poset $P$, the action groupoid $G \ltimes P$ has $(G \ltimes P)_0 = P$ and the action of $G$ on $P$ derived from the
covering $G \ltimes P \ra G$ is the given action of $P$ on $G$.

For a fixed ordered groupoid $G$, let $\act(G)$ be the category whose objects are $G$--posets and whose morphisms are
maps of $G$--posets, and let $\cov(G)$ be the category whose objects are ordered coverings $C \ra G$ and in which a morphism
from $\gamma: C \ra G$ to $\delta: C' \ra G$ is an ordered functor $\phi: C \ra C'$ such that $\phi \delta = \gamma$.
Then the correspondences between coverings and poset actions just outlined give the following result.

\begin{prop}
\label{cov_vs_act}
For a fixed ordered groupoid $G$, the categories $\act(G)$ and $\cov(G)$ are naturally equivalent.
\end{prop}

\subsection{Constructing a covering from an immersion}
Let $\phi: U \ra H$ be an immersion of ordered groupoids.  Our aim is to factorize $\phi$ through a covering
$\twH \stackrel{\pi}{\ra} H$, and by Proposition \ref{cov_vs_act} this amounts to constructing
an $H$--action on a suitable poset.   We may regard the set of arrows of $H$ as a poset, which we denote by  $\Omega H$: we justify this notation in section \ref{fibthm}. 
(We note that $\Omega H$ is denoted by $H_D$ in \cite{St}.)  We now construct the free right $H$--poset on $U_0$: this is the 
pullback
\[ U_0 \boxtimes \Omega H = \{ (e,h) : e\phi = h \dom \} \]
ordered componentwise, and there is an obvious right $H$--action by multiplication, so we can form the
semidirect product
\[ H \ltimes (U_0 \boxtimes \Omega H)  = \{ (k,(e,h)) : e\phi = h \dom, h \ran = k \ran \}. \]
As above, we identify the object set of the semidirect product with $U_0 \boxtimes H$: then we have
\[ (k,(e,h)) \dom = (e,hk^{-1}) \; \text{and} \; (k,(e,h))\ran = (e,h) \]
and the composition of $(k,(e,h))$ and $(m(f,l))$, defined when $(e,h)=(f,lm^{-1})$, is given by
\[ (k,(e,h))(m,(f,l)) = (km,(f,l)).\]
We now wish to define an ordered embedding $\iota: U \ra H \ltimes (U_0 \boxtimes \Omega H)$, but the obvious candidate
mapping $u \mapsto (u\phi,(u \dom, u \phi))$ is not a functor: $(u\phi,(u \dom, u \phi))\ran = (u \dom,u\phi)$
whilst $(u \ran)\iota = (u \ran, u \ran \phi)$.  Similar problems arise with variations upon this idea. The remedy is to
impose an equivalence relation $\equiv_U$ on $U_0 \boxtimes \Omega H$:
\[ (e,h) \equiv_U (f,k) \Longleftrightarrow \; \text{there exists} \; u \in U(e,f) \;\text{such that}\; h=(u\phi)k. \]
It is clear that $\equiv_U$ is an equivalence relation, and that for all $u \in U$, $(u \dom,u\phi) \equiv_U
(u \ran,u\ran \phi)$.  We denote the equivalence class of $(e,h)$ by $e \otimes h$ and the quotient set
$(U_0 \boxtimes \Omega H)/\equiv_U$ by $U_0 \otimes \Omega H$.

\begin{lemma}
\label{tensor_is_poset}
The quotient set $U_0 \otimes \Omega H$ is a poset on which $H$ acts on the right.
\end{lemma}

\begin{proof}
The ordering on $U_0 \otimes \Omega H$ is induced from the componentwise ordering on the pullback $U_0 \boxtimes  \Omega H$: we have
\[ f \otimes h \leq e \otimes k \; \Longleftrightarrow \; f \otimes h = x \otimes l \; \text{with} \; x \leq e \;
\text{and} \; l \leq k \,.\]
We show that this is a well-defined relation on $U_0 \otimes  \Omega H$. If $f \otimes h \leq e \otimes k$ as above, then there
exists $v \in U(f,x)$ such that $h=(v\phi)l$, and if
$e \otimes k = e' \otimes k'$ then there exists $u \in U(e,e')$ such that $k=(u\phi)k'$.  Set $y=(x|u)\ran$
and $k'' = (y\phi|k')$.
Then $(x\phi|u\phi)k'' \leq (u \phi)k' = k$ and $(x\phi|u\phi)\dom = x \phi = l \dom$.  But $l \leq k$,
and so by uniqueness of restriction, we have $l = (x\phi|u\phi)k''$.  Now
\begin{align*}
e \otimes k = e' \otimes k' & \geq y \otimes k'' = f \otimes (v \phi)(x|u)\phi k'' \\
&= f \otimes (v \phi)(x \phi|u \phi)k'' \\
&= f \otimes (v\phi)l = f \otimes h.
\end{align*}
It is now clear that $\leq$ is reflexive and transitive on $U_0 \otimes H$.  To show that it is anti-symmetric,
suppose that $e \otimes k \leq f \otimes h \leq e \otimes k$.  Then there exist $x \leq f$ and $u \in U(x,e)$ such that
$(u \phi)k \leq h$, and $e \otimes k = x \otimes (u \phi)k$.  Hence we may as well assume that $e \leq f$ and $k \leq h$.
Since $F \otimes h \leq e \otimes k$ there exist $y \leq e$ and $v \in U(y,f)$ such that $(v \phi)h \leq k$.  So
\[ y \leq e \leq f \; \text{and} \; (v \ph)h \leq k \leq h \,. \]
Since $(v(\phi)h)\ran=h \ran$ we deduce that $(v\phi)h=h$ and so $v \phi = h \dom = f \phi$.  Since $\phi$ is star-injective, $v=f$ and so
$h=k$, $y=f$ and so finally $e=f$.

The right action of $H$ on $U_0 \otimes  \Omega H$ is defined by right multiplication: if $f \otimes h = x \otimes l$
then $f\ran = l\ran$ and so for $m \in H$ we may define $(f \otimes h) \lhd m = f \otimes hm$.
\end{proof}

We now form the semidirect product $\twH = H \ltimes (U_0 \otimes  \Omega H)$.  We have an ordered
embedding $i: U \ra \twH$ mapping $u \mapsto (u \phi, u \dom \otimes u \phi)$ and a covering
$\pi : \twH \ra H$ such that $\phi = i \pi$.  In fact, $i$ has additional properties.
 An
{\em enlargement} of ordered groupoids is an inclusion of a subgroupoid $A \hookrightarrow B$ such that
\begin{itemize}
\item $A_0$ is an order ideal in $B_0$,
\item if $b \in B$ and $b \dom, b \ran \in A$ then $b \in A$,
\item if $e \in B_0$ then there exists $b \in B$ with $b \dom = e$ and $b \ran \in A_0$.
\end{itemize}

\begin{theorem}[Maximum enlargement \cite{Ehres,LwBook}]
Let $\phi: U \ra H$ be a star-injective ordered functor.
The semidirect product $\twH$ is an enlargement of $U$, such that for any factorization of $\phi$
as the composition $U \stackrel{j}{\ra} C \stackrel{\xi}{\ra} H$ of an ordered embedding and an ordered
covering, there exists a unique ordered functor $\nu : \twH \ra C$ such that $j=i \nu$ and
$\pi = \nu \xi$.
\end{theorem}

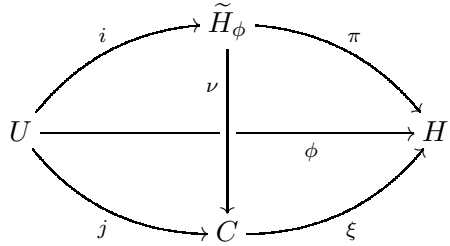
\[  \xymatrixcolsep{5pc}
\xymatrix{
& \twH \ar@/^1pc/[rd]^{\pi} \ar[dd]_(.3){\nu} & \\
U \ar@/^1pc/[ur]^{i} \ar[rr]_(.7){\phi}|(.5)\hole \ar@/_1pc/[dr]_j && H \\
& C \ar@/_1pc/[ur]_{\xi} &
} \]

\begin{proof}
We first show that $\twH$ is an enlargement of $Ui$.  To show that $(Ui)_0$ is an order ideal in
$(\twH)_0$, suppose that $f \otimes h \leq u\dom \otimes u\phi$.  Then $f \otimes h = x \otimes l$ with
$x \leq u \dom$ and $l \leq u\phi$.  Since $x\phi = l\dom$ it follows that $l = (x\phi|u \phi)=(x|u)\phi$
and hence that $f \otimes h = x \otimes (x|u)\phi \in (Ui)_0$.

Now suppose that $e \otimes k = (h,e \otimes k)\dom = u\dom \otimes u \phi$ and that
$e \otimes kh^{-1} = (h,e \otimes k)\ran = v\dom \otimes v \phi$ for some $u,v \in U$.  Hence there exist
$a \in U(u\dom,e)$ such that $u\phi = (a\phi)k$ and $b \in U(v\dom,e)$ such that $v\phi=(b\phi)kh^{-1}$.
Then 
\[ h = (v\phi)^{-1}(b\phi)k= (v\phi)^{-1}(b\phi)(a \phi)^{-1}(u \phi) = (v^{-1}ba^{-1}u)\phi \,.\]
Now $i: v^{-1}ba^{-1}u \mapsto (h,v\ran \otimes h)$, but 
\[ v\ran \otimes h = v\ran \otimes (v^{-1}ba^{-1}u)\phi = e \otimes (a^{-1}u)\phi = e \otimes k \]
and so $(h,e \otimes k) \in Ui$.

Now given $e \otimes k \in (\twH)_0$ the arrow $(k^{-1},e \otimes k\dom)$ has domain $e \otimes k$ and range
$e \otimes k \dom = e \otimes e \phi \in (Ui)_0$.

This completes the verification that $\twH$ is an enlargement of $Ui$.

Now suppose that we are given a factorization of $\phi$
as the composition $U \stackrel{j}{\ra} C \stackrel{\xi}{\ra} H$ of an ordered embedding and an ordered
covering.  For $(h,e \otimes k) \in \twH$ we have $k \in \st_H(e\phi) = \st_H(ej \xi)$, and since $\xi$ is a covering
there exists a unique $c \in \st_C(ej)$ with $c \xi = k$.  Let $y = c \ran$: then $y \xi = k \ran = h \ran$ and so
there exists a unique $q \in C$ with $q^{-1} \in \st_C(y)$ and $q \xi  = h$.  We define $(h,e \otimes k)\nu=q$.

This is well-defined since, if $e \otimes k = f \otimes l$ then there exists $u \in U(e,f)$ with $(u\phi)l=k$.  Then
the unique element of $\st_C(x)$ mapping to $l$ is $(uj)^{-1}c$ and we obtain $y = c\ran = ((uj)^{-1}c)\ran$ and hence
$q$ as before.

To show that $\nu$ is a functor, consider a composition $(h,e \otimes k)(l, e \otimes kl) = (hl, e \otimes kl) \in \twH$,
with $c,y,q$ defined as above, and $(H, e \otimes k)\nu = q$.  There exists a unique $d \in \st_C(y)$ with $d \xi = l$.
Then $(l,e \otimes kl)\nu=d$ and $(hl,e \otimes kl) = qd$, and so $\nu$ is a functor.  If $(h,e \otimes k) \leq (l,f \otimes m)$ then we have
$h \leq l$ and we may assume that $e \leq f$ and $k \leq m$.  We find a unique $s \in \st_C(fj)$ with $s \xi = m$.  Then 
$(ej|s)\xi \dom = e \phi$ and $(ej|s)\xi \leq s \xi = m$: hence $(ej|s)=c$.  Now there exists $t^{-1} in \st_C(s \ran)$ with $t\xi = l$ and
$(l, f \otimes m)\nu=t$.  Since $c \leq s$ we have $y \leq s \ran$ and $(t|y)\xi \leq l$.  Hence $(t|y)\xi=h$ and so $q = (t|y) \leq t$.  Hence
$\nu$ is an ordered functor.

To show that $\nu$ is unique with the stated properties, suppose that $\mu$ also possesses them.  Then for $(h, e \otimes k) \in \twH$,
we have
\[ h = (h, e \otimes k)\pi = (h, e \otimes k)\nu \xi  = (h, e \otimes k)\mu \xi \]
and since $\xi$ is a covering, if $\mu$ and $\nu$ agree on the identity  $e \otimes kh^{-1}$  then $(h,e \otimes k)\nu = (h,e \otimes k) \mu$.  But $\mu$ and $\nu$ agree on $(\twH)_0$:
for since $\iota \nu = j = \iota \mu$ we have that $\mu = \nu$ on $U \iota$, and then for any $z \in (\twH)_0$ we can join
$z$ to an $x \in (U\iota)_0$ by some $a \in \twH$.  But then $a \mu = a \nu$ and so $z \mu = z \nu$.
\end{proof}

\begin{cor}
Any ordered functor $\phi : G \ra H$ admits a canonical factorization as the composition of a fibration, followed by
an enlargement, followed by a covering.
\[ G \ra G \twobar \ker \phi \ra \twHpsi \ra H \,. \]
\end{cor}

\section{The fibration theorem and the derived ordered groupoid}
\label{fibthm}
In \cite{St} Steinberg introduced a construction that he called the {\em derived ordered groupoid}
$\der(\phi)$ of a mapping $\phi : G \ra H$ of ordered groupoids.  Amongst its applications is the 
{\em Fibration Theorem} \cite[Theorem 5.1]{St}:
any mapping of ordered groupoids can be factorized as the composition of an enlargement followed by a fibration. 
The following approach to $\der(\phi)$ is based on a construction from topology -- the {\em mapping cocylinder} -- that is used to prove a similar
theorem: every continuous map of topological spaces can be factorised as the composition of a homotopy equivalence followed by a 
fibration (see for example \cite[Theorem I.7.30]{WhBook}, where the construction is called the {\em mapping path space}.).  Steinberg constructs the mapping cocylinder indirectly, as a semidirect product $H \ltimes \der(\phi)$.
We give a direct account of Steinberg's factorization, following but slightly simplifying the formulation in \cite{LMP}.

We start with the mapping groupoid $\ogpd(\I,H)$.  Recall that this is a groupoid
whose objects are the groupoid maps $\I \ra H$: however, any such map is completely determined by the image of the arrow $\iota$
of $\I$, and so we may identify the objects of $\ogpd(\I,H)$ with the poset $\Omega H$ of arrows of $H$.  An arrow in
$\ogpd(\I,H)$ is  a natural transformation between maps $\I \ra H$.  Identifying two such maps with arrows $h_0, h_1 \in \Omega H$, a
natural transformation from $h_0$ to $h_1$ is then a pair of arrows $(t_0,t_1)$ of $H$ such that the square
\begin{equation}
\label{commsq}
\xymatrixcolsep{3pc}
\xymatrix{
\bullet \ar[d]_{t_0} \ar[r]^{h_0} & \bullet \ar[d]^{t_1}\\
\bullet \ar[r]_{h_1} & \bullet}
\end{equation}
commutes.  Now any three of the arrows in the square determines the fourth.  We choose to suppress $t_1$: rewriting $t_0$ as $t$, the arrow in $\ogpd(\I,H)$  in \eqref{commsq} will be
identified with the triple $[h_0,t,h_1]$ of arrows in $H$, which are subject to the condition that the composition $h_0^{-1}th_1$ exists.
Then we have
\[ [h_0,t,h_1] \dom =h_0, \; [h_0,t,h_1] \ran = h_1 \,, \]
and the composition of $[h_0,t,h_1]$ and $[h_1,u,s_1]$ is defined to be
\[ [h_0,t,h_1][h_1,u,s_1] = [h_0,tu,s_1].\]

\begin{prop}
\label{ep0isfib}
For $i=1,2$ the ordered functors $\ep_i: \ogpd(\I,H) \ra H$ given by 
\begin{equation*}
\ep_0: \begin{cases} h  \mapsto h \dom & \text{if $h \in \Omega H$}\\ [h_0,t,h_1] \mapsto t & \end{cases}
\end{equation*}
and
\begin{equation*}
\ep_1: \begin{cases} h  \mapsto h \ran & \text{if $h \in \Omega H$}\\ [h_0,t,h_1] \mapsto h_o^{-1}th_1 & \end{cases}
\end{equation*}
are strong fibrations.
\end{prop}

\begin{proof}
We first consider $\ep_0$.  Suppose that, in a commutative square
\[ \xymatrixcolsep{3pc}
\xymatrix{
A \ar[d]_{i_0} \ar[r]^{f} & \ogpd(\I,H) \ar[d]^{\ep_0}\\
A \times \I \ar[r]_{F} & H}
\]
we have, for $x \in A_0$, $xf=h_x \in \Omega H$. and that if $a \in A$ with $a \dom =x$ and $a \ran = y$ then $af=[h_x,t_a,h_y]$. 
Suppose further that $(y,\iota)F=l_y$.  Then
\[ l_y \dom = (y,\iota)F \dom = (y,\iota)\dom F = (y,0)F = yf\ep_0 = h_y \dom = t_a \ran. \]
Define $\widetilde{F_0}$ by:
\begin{equation*}
\widetilde{F_0}:
\begin{cases}
(x,0)  \mapsto h_x \\
(y,1)  \mapsto l_y \ran \\
(a,\iota)  \mapsto [h_x,t_al_y,l_y \ran].  \end{cases}
\end{equation*}
Certainly this makes
\[ \xymatrixcolsep{3pc}
\xymatrix{
A \ar[d]_{i_0} \ar[r]^{f} & \ogpd(\I,H) \ar[d]^{\ep_0}\\
A \times \I \ar[r]_{F} \ar[ur]_{\widetilde{F_0}} & H}
\]
commute, and it is easy to check that this defines an ordered functor.

The argument for $\ep_1$ is similar.  In a commutative square
\[ \xymatrixcolsep{3pc}
\xymatrix{
A \ar[d]_{i_0} \ar[r]^{f} & \ogpd(\I,H) \ar[d]^{\ep_1}\\
A \times \I \ar[r]_{F} & H}
\]
with $af=[h_x,t_a,h_y]$ as above, and with $(y,\iota)F=l_y$, we have
\[ l_y \dom = (y,\iota)F \dom = (y,\iota)\dom F = (y,0)F = yf\ep_1 = h_y \ran. \] 
leading to the definition of a homotopy lifting $\widetilde{F_1} : A \times \I \ra \ogpd(\I,H)$
by $(a,\iota)\widetilde{F_1} = [h_x,t_a,h_yl_y]$.
\end{proof}

Given $\phi : G \ra H$ its {\em mapping cocylinder} $M^{\phi}$ is the pullback of the diagram
\[ \xymatrixcolsep{3pc}
\xymatrix{
 & \ogpd(\I,H) \ar[d]^{\ep_0}\\
G \ar[r]_{\phi} & H}
\]
that is
\[ M^{\phi} = \{ (a,[h_0,a\phi,h_1]) : a \in G, h_0,h_1 \in H, \; h_0^{-1}(a\phi)h_1 \; \text{exists} \,\}. \]
We abbreviate the $4$--tuple $(a,[h_0,a\phi,h_1])$ to the triple $\<h_0,a,h_1\>$ so that now
\[ M^{\phi} = \{ \<h_0,a,h_1\> : a \in G, h_0,h_1 \in H, h_0^{-1}(a \phi)h_1 \; \text{exists} \}. \]
Then $M^{\phi}$ is an ordered groupoid with $M^{\phi}_0$ equal to the pullback
\[ \{ (e,h) \in G_0 \times H : e \phi = h \dom \} \,. \]
We have
\[ \<h_0,a,h_1\> \dom = (a \dom, h_0), \; \<h_0,a,h_1\>\ran = (a \ran, h_1) \]
and if $(a \ran, h_1) = (b \dom,k_0)$ then the composition of $\<h_0,a,h_1\>$ and
$\<k_0,b,k_1\>$ is given by
\[ \<h_0,a,h_1\>\<k_0,b,k_1\> = \< h_0,ab,k_1 \>. \]
The ordering on arrows in $M^{\phi}$ is componentwise:
\[ \<h_0,a,h_1\> \leq \<k_0,b,k_1\> \; \iff \; h_0 \leq k_0, a \leq b, \; \text{and} \; h_1 \leq k_1 \,.\]

We can now complete our account of Steinberg's Fibration Theorem \cite[Theorem 5.1]{St}, which follows that given
in \cite[Theorem 4.9]{LMP}.

\begin{theorem}
\label{st_fib_thm}
 Let $\phi : G \ra H$ be a morphism of ordered groupoids.  Then $\phi$ admits a factorization
\[ \xymatrixcolsep{3pc}
\xymatrix{G \ar[r]^{i_{\phi}} \ar@/_2pc/[rr]_{\phi} & M^{\phi} \ar[r]^{p_{\phi}} & H}
\]
where $i_{\phi} : g \mapsto \< (g \dom) \phi ,g,(g\ran) \phi \>$ and $M^{\phi}$ is an enlargement of
$Gi_{\phi}$, and $p_{\phi} : \< h_0,a,h_1 \> \mapsto h_0^{-1}(a \phi)h_1$ is a strong fibration.
\end{theorem}

\begin{proof}
It is clear that $i_{\phi}$ and $p_{\phi}$ are morphisms of ordered groupoids and that $\phi = i_{\phi}p_{\phi}$.

Suppose we are given a commutative square
\[ \xymatrixcolsep{3pc}
\xymatrix{
A \ar[d]_{i_0} \ar[r]^{f} & M^{\phi} \ar[d]^{p_{\phi}}\\
A \times \I \ar[r]_{F} & H}
\]
For an arrow $a \in A$, let $af= \< h_a,g_a,k_a \> \in M^{\phi}$ so that $(a,0)F = h_a^{-1}(g_a\phi)k_a$.
Suppose that $(a \ran,\iota)F = l_{a \ran} \in H$. Then 
\[ l_{a \ran} \dom = (a \ran, \iota)F \dom 
= (a \ran,\iota)\dom F  = (a \ran,0)F = (a \ran)f p_{\phi} = (afp_\phi) \ran = k_a \ran \,. \]
Hence $k_a$ and $l_{a \ran}$ are composable arrows in $H$, and we define
\[ (a,\iota) \widetilde{F} = \< h_a,g_a,k_al_{a \ran} \> \,. \]
This is an ordered groupoid morphism that lifts $F$.

Now $Gi_{\phi} = \{ \< (g \dom) \phi ,g,(g\ran) \phi \> : g \in G \}$ is a subgroupoid of $M^{\phi}$,
with $(Gi_{\phi})_0 = \{ (e,e\phi) : e \in G_0 \}$.  Suppose that $(x,h) \in M^{\phi}_0$ and that
$(x,h) \leq (e,e\phi)$.  Then $x \phi = h \dom, x \leq e$ and $h \leq e\phi$.  The last condition implies that
$h \in H_0$, whence $h=x \phi$ and so $(x,h) \in (Gi_{\phi})_0$.  We have shown that $(Gi_{\phi})_0$ is an order ideal
in $M^{\phi}_0$.

Now suppose that $\<h_0,a,h_1\> \in M^{\phi}$ and that $(a \dom,h_0)=(x,x\phi)$ and $(a\ran,h_1)=(y,y\phi)$
with $x,y \in G_0$.  Then $\<h_0,a,h_1\>=\< (a\dom)\phi,a,(a \ran)\phi \> \in Gi_{\phi}$, which establishes 
the second condition for an enlargement.  Finally, take $(e,h) \in M^{\phi}_0$.  Then $\< h,e,e\phi\>\dom =(e,h)$ and 
$\<h,e,e\phi\>\ran = (e,e\phi) \in (Gi_{\phi})_0.$.  This concludes the proof that $M^{\phi}$ is an enlargement
of $Gi_{\phi}$.
\end{proof}

\begin{cor}
 Let $\phi : G \ra H$ be a morphism of ordered groupoids.  Then $\phi$ admits a factorization
\[ \xymatrixcolsep{3pc}
\xymatrix{G \ar[r]^{i_{\phi}} \ar@/_3pc/[rrr]_{\phi} & M^{\phi} \ar[r]^(.4){\varpi} & M^{\phi} \twobar \ker p_{\phi} \ar[r] & H}
\]
 as an enlargement followed by a strong fibration followed by a covering.
\end{cor}

\begin{proof}
Use Corollary \ref{fib_as_fib_cov} to factorize $p_{\phi}$ through $M^{\phi} \twobar \ker p_{\phi}$.  The quotient map
$\varpi: M^{\phi} \ra M^{\phi} \twobar \ker p_{\phi}$ is a strong fibration by Proposition \ref{fib_and_immer}.
\end{proof}

\begin{cor}
\label{ker_is_der}
 The kernel of the fibration $p_{\phi}$ is equal to Steinberg's derived ordered groupoid $\der(\phi)$.
\end{cor}

\begin{proof}
We have
\begin{align*}
\ker p_{\phi} &= \{ \<h_0,a,h_1\>: h_0^{-1}(a \phi)h_1 \in H_0 \} \\
&= \{ \<h_0,a,h_1\> : h_0=(a\phi)h_1 \}.
\end{align*}
We can suppress mention of $h_0$ in elements of $\ker p_{\phi}$, and so write
$\ker p_{\phi} = \{ (a,h) : (a \phi)\ran=h \dom \}$.
In this notation $(\ker p_{\phi})_0 = \{ (e,h) : e\phi=h\dom \}$ with
$(a,h)\dom = (a\dom, (a \phi)h)$ and $(a,h)\ran = (a\ran,h)$.  The composition of
$(a,h)$ and $(b,k)$, defined when $(a \ran,h) = (b \dom, (b \phi)k)$, is then given by
$(a,h)(b,k)=(ab,k)$.  With due allowance for the change of notation required by Steinberg's use of
left actions of groupoids, this is precisely $\der(\phi)$.
\end{proof}

Now $H$ acts on $\der(\phi)$: we have $\omega: (a,h) \mapsto h \ran$, so that $(a,h) \lhd h'$ is defined when $hh'$ exists, and then
$(a,h) \lhd h' = (a,hh')$.

\begin{cor}[\cite{LMP}, Proposition 4.9]
\label{M_is_sdp}
The mapping cocylinder $M^{\phi}$ is isomorphic to the semidirect product $H \ltimes \der(\phi)$.
\end{cor}

\begin{proof}
 We define $\gamma: H \ltimes \der(\phi) \ra M^{\phi}$ by $(k,(g,h)) \mapsto \< (g \phi)hk^{-1},g,h \>$.
Now the composition of $(k,(g,h))$ and $(l,(p,q))$ in $H \ltimes \der(\phi)$ is defined if and only if
$k \ran = l \dom, g \ran = p \dom$ and $h=(p \phi)ql^{-1}$, and then we have
\[ (k,(g,h))(l,(p,q)) = (kl,(g,hl)(p,q)) = (kl,(gp,q)). \]
Then
\begin{align*}
(k,(g,h)) \gamma (l,(p,q)) \gamma &= \< (g \phi)hk^{-1},g,h \> \< (p \phi)ql^{-1},p,q\> \\
&= \< (g \phi)hk^{-1},g,h \> \< h,p,q\> \\
&= \< (g \phi)hk^{-1},gp,q \> = \< (g \phi)(p \phi)ql^{-1}k^{-1},gp,q \> \\
&= (kl,(gp,q)) \gamma.
\end{align*}
It follows that $\gamma$ is a functor, and $\gamma$ is then easily seen to be an isomorphism of ordered groupoids, with inverse
\[ \gamma^{-1} : \< h_0,a,h_1 \> \mapsto (h_0^{-1}(a \phi)h_1, (a,h_1))\,.\]
\end{proof}

\begin{remark}
Under the isomorphism of Corollary \ref{M_is_sdp}, the mapping $p_{\phi}: M^{\phi} \ra H$ corresponds to the projection
$\pi : H \ltimes \der(\phi) \ra H$ and the lifting $\widetilde{F}$ described in the proof of Theorem \ref{st_fib_thm} is the same as that
given in Lemma \ref{sdp_fib}.
\end{remark}

Take $\phi$ to be the inclusion $H_0 \hookrightarrow H$.  Then $\der(\phi)$ will be the analogue of the loop
space $\Omega H$ of $H$. We then have
\[ \Omega H = \der(\phi) = \{ (h \dom,h) : h \in H \} \]
and so $\Omega H$ is identified as the trivial groupoid on the set of arrows of $H$, ordered as a partially ordered
set by the ordering on $H$.  This is the groupoid $H_D$ of \cite{St}.

\begin{prop}
\label{loops_on_H}
The poset $\Omega H$ is naturally isomorphic to $\ker \ep_0 \cap \ker \ep_1$.
\end{prop} 

\begin{proof}
The kernel of $\ep_0$ is the analogue of the set of pointed maps from the unit interval
(based at $0$) to a pointed space $(X,x_0)$.  We have
\[ \ker \ep_0 = \{ [h_0,e,h_1]: e \in H_0, h_0^{-1}eh_1 \; \text{exists} \} \]
and so $\ogpd_*(\I,H)$ may be identified with the groupoid whose arrows are pairs of coinitial arrows in $H$,
and a pair $(h_0,h_1) \in \ogpd_*(\I,H)$ then satisfies $(h_0,h_1)\dom = h_0, (h_0,h_1)\ran=h_1$ with composition
$(h_0,h_1)(h_1,h_2)=(h_0,h_2)$.  Then 
\[ \ep_1: \ker \ep_0 \ra H \,, (h_0,h_1) \mapsto h_0^{-1}h_1 \,. \]
whose kernel is the subset $\{ (h,h) : h \in H \}$ with $(h,h)\dom = h = (h,h)\ran$ and $(h,h)(h,h)=(h,h)$:
that is $\Omega H$.
\end{proof}

\begin{prop}
The projection $q_{\phi}:M^{\phi} \ra G$ defined by 
$q: \<h_0,a,h_1 \> \mapsto a$ is a strong fibration, and its restriction to $\der(\phi)$ is a covering. 
\end{prop}

\begin{proof}
The map $q_{\phi}$ is a pullback of $\ep_0: \ogpd(\I,H) \ra H$ and so is a strong fibration by Propositions \ref{ep0isfib} and 
\ref{cov_is_strong}.
Using the notational changes introduced in the proof of Corollary
\ref{ker_is_der} we have $q_{\phi}: \der(\phi) \ra G$ given by $(a,h) \mapsto a$, and hence 
\[
 \ker q_{\phi} = \{ (e,h) : e \in G_0, e \phi = h \dom \} = (\der(\phi))_0 \,. 
\]
\end{proof}


\begin{thebibliography}{99}
\bibitem{Br} R. Brown, Fibrations of groupoids.  J. Algebra 15 (1970) 103--132.
\bibitem{Br2} R. Brown, Groupoids as coefficients.  Proc. London Math Soc. (3) 25 (1972) 413--426.
\bibitem{BHK} R. Brown, P.R. Heath, and K.-H. Kamps, Coverings of groupoids and Mayer-Vietoris type sequences. In
\emph{Categorical Topology, Proc. Conf. Toledo Ohio 1983}  Sigma Ser. Pure Math 5 147–162, Heldermann, Berlin (1984).
\bibitem{Ehres} C. Ehresmann, \emph{Oeuvres Compl\`{e}te et Comment\'{e}es.} (A.C.Ehresmann ed.)
Supplements to Cah. Topologie G\'{e}om. Diff\'{e}r. Cat\'{e}goriques (1980-84).
\bibitem{Gi3} N.D. Gilbert, A $P$--theorem for ordered groupoids. In {\em Proc. Intl. Conf.
Semigroups and Formal Languages, Lisbon 2005} J.M Andr\'e et al. (Eds.) 84-100. World Scientific (2007).
\bibitem{HiBook} P.J. Higgins, {\em Notes on categories and
groupoids}. Van Nostrand Reinhold Math. Stud. 32 (1971).  Reprinted
electronically at  www.tac.mta.co/tac/reprints/articles/7/7tr7.pdf\,.
\bibitem{LwBook} M.V. Lawson, {\em Inverse Semigroups.}  World Scientific (1998).
\bibitem{LMP} M.V. Lawson, J. Matthews and T. Porter, The homotopy theory of inverse semigroups.
Internat. J. Algebra Comput. 12 (2001) 755--790.
\bibitem{Mthesis} J. Matthews, {\em Topological Ideas in Inverse Semigroup Theory.}  PhD Thesis, 
University of Wales, Bangor,  (2004).
\bibitem{ecm_phd} E.C. Miller, \emph{Structure Theorems for Ordered Groupoids}.  PhD Thesis, Heriot-Watt University,
Edinburgh, (2009).
\bibitem{Rei} N.R. Reilly, Bisimple $\omega$--semigroups.  Proc. Glasgow Math. Assoc. 7 (1966) 160-167.
\bibitem{St} B. Steinberg, Factorization theorems and morphisms of ordered groupoids and inverse semigroups,
Proc. Edin. Math. Soc. 44 (2001) 549-569.
\bibitem{Wei} A. Weinstein, Groupoids: unifying internal and external symmetry.  Notices Amer. Math. Soc. 43 (1996) 744-752.
\bibitem{WhBook}  G.W. Whitehead, {\em Elements of Homotopy Theory.}  Grad. Texts in Math. 61, Springer-Verlag (1978).
\end{thebibliography}
\end{document}